%% file: main.tex
\documentclass[12pt]{amsart}
\usepackage{amssymb}
\usepackage{amsmath}
\usepackage{amsfonts}
\usepackage{tikz-cd}
\usepackage{hyperref}
\hypersetup{
  hidelinks,
  pdftitle={Cellular A1-homology of wonderful models of subspace arrangements},
  pdfauthor={Haoyang Liu and Keyao Peng}
}
\providecommand{\llbracket}{\mathopen{[\mkern-3mu[}}
\providecommand{\rrbracket}{\mathclose{]\mkern-3mu]}}
\newtheorem{theorem}{Theorem}[section]
\newtheorem{lemma}[theorem]{Lemma}
\newtheorem{proposition}[theorem]{Proposition}
\newtheorem{corollary}[theorem]{Corollary}
\newtheorem{definition}[theorem]{Definition}

\theoremstyle{definition}
\newtheorem{remark}[theorem]{Remark}
\newtheorem{example}[theorem]{Example}

\numberwithin{equation}{section}
\allowdisplaybreaks[1]

\newcommand{\A}{\mathbb{A}}
\newcommand{\Z}{\mathbb{Z}}
\newcommand{\R}{\mathbb{R}}
\newcommand{\C}{\mathbb{C}}
\newcommand{\Q}{\mathbb{Q}}

\newcommand{\G}{\mathcal{G}}
\newcommand{\PP}{\mathbb{P}}

\newcommand{\F}{\mathbb{F}}

\newcommand{\KMW}{\mathrm{K}^\mathrm{MW}}

\newcommand{\Spec}{\mathrm{Spec}\ }

\newcommand{\afnz}[1]{\mathbb{A}^{#1}\setminus \{0\}}

\newcommand{\bZ}{\mathbb{Z}}

\newcommand{\Gm}{\mathbb{G}_m}

\newcommand{\llBra}[1]{\llbracket #1 \rrbracket}

\newcommand{\xr}[1]{\xrightarrow{#1}}
\newcommand{\Build}{\mathbf{Build}}
\newcommand{\Vect}{\mathbf{Vect}}
\newcommand{\Dor}{\mathrm{Dor}}

\newcommand{\cF}{\mathcal{F}}

\newcommand{\cN}{\mathcal{N}}
\newcommand{\cU}{\mathcal{U}}

\newcommand{\bS}{\mathbb{S}}
\newcommand{\AZ}{\mathbb{A}\mathcal{Z}}

\title[Cellular $\A^1$-homology of wonderful models]
{Cellular $\A^1$-homology of wonderful models of subspace arrangements}
\author{Haoyang Liu}
\address{Department of Mathematics, University of California, Santa Barbara,
CA 93106, USA}
\email{haoyangliu@ucsb.edu}
\author{Keyao Peng}
\address{Universit\'e Bourgogne Europe, Dijon, France}
\email{keyao.peng@ube.fr}

\begin{document}
\begin{abstract}
We compute the cellular $\A^1$-homology of De Concini--Procesi
wonderful models of subspace arrangements.  For a building set $\G$
over a field $k$, we identify the cellular $\A^1$-chain complex of
$\PP(\G)$ with an $\eta$-twisted nested-set complex carrying
Milnor--Witt coefficients and derived orientation data.  The key
geometric input is a motivic blow-up calculation: for a blow-up along a
smooth center of codimension $c$, the relevant connecting class is
$(c-1)_\epsilon\eta$, hence it is zero for $c$ odd and $\eta$ for $c$
even.  This replaces the parity condition in the computation of Rains
by a Milnor--Witt attaching class.  As a consequence, the part of
cellular $\A^1$-homology surviving after multiplication by $\eta$, and
also the homology after inverting $\eta$, are expressed by the interval
cohomology of the $2$-divisible subposet of the lattice generated by
$\G$.  For the braid arrangement, the condition becomes the odd-block
condition on partitions, yielding explicit decompositions for the
cellular $\A^1$-homology of $\overline{\mathcal M}_{0,N}$ and examples
in low rank.
\end{abstract}
\maketitle
\tableofcontents

\input{intro}
\input{combinatorial}
\input{blow-up}
\input{cellular}
\input{example}

\bibliographystyle{plain}
\bibliography{references}
\end{document}

%% file: intro.tex
\section{Introduction}

Wonderful compactifications of subspace arrangement complements were
introduced by De Concini and Procesi as smooth projective models whose
boundary strata are controlled by nested sets in a building set
\cite{CP}.  They include several familiar spaces: for example, the
braid arrangement yields the moduli space \(\overline{\mathcal M}_{0,n+1}\).
The topology of these models is remarkably combinatorial.  In the real
case, Rains computed the homology of the real locus by using the
nested-set stratification and showed that the answer is governed by a
parity condition on the dimensions of the nodes of the forest
\cite{Rains_2010}.

The purpose of this paper is to lift this picture to cellular
\(\A^1\)-homology in the sense of Morel and Sawant
\cite{MOREL2023109346}.  More precisely, for a building set \(\G\) in
\(V^*\) over a field \(k\), we compute the cellular \(\A^1\)-chain
complex of the De Concini--Procesi model \(\PP(\G)\) in terms of a
purely combinatorial nested-set complex with Milnor--Witt coefficients.
The resulting answer is a motivic refinement of Rains' real
calculation: the classical parity phenomenon is replaced by the
Milnor--Witt attaching class coming from a blow-up, and the surviving
part after multiplication by \(\eta\) is controlled by the
\(2\)-divisible part of the lattice of sums of elements of \(\G\).

The strategy is parallel to the toric calculation of
\cite{liu2025cellularmathbba1homologysmoothtoric}.  There a smooth fan
is encoded by its simplicial complex \(K\) and characteristic map
\(\lambda\), and the cellular differential is recovered from a
moment-angle quotient together with the orientation changes determined
by \(\lambda\).  Here the nested-set complex and its derived
orientation coefficients play the corresponding roles.  The
permutohedral toric variety is the type-\(A\) overlap of the two
frameworks; Section~2 explains both this comparison and the distinction
from the braid-arrangement model used for
\(\overline{\mathcal M}_{0,n+1}\).

Our first ingredient is a functorial reformulation of wonderful models.
We regard a building set as a diagram of quotient vector spaces and
realize \(\PP(\G)\) as the space of sections of the projective
geometric context.  This point of view makes the projection maps,
operadic boundary maps, and the idempotent decomposition indexed by
\(\mathcal C_{\G}\) transparent.  It also gives the bookkeeping needed
to compare blow-ups on the geometric side with exact sequences of
nested-set complexes on the combinatorial side.

The second ingredient is the motivic blow-up calculation.  If
\(\widetilde X=\operatorname{Bl}_Z X\) is the blow-up of a smooth
scheme along a smooth center of codimension \(c\), then the top
cellular component on an oriented normal slice is multiplication by
\[
  (c-1)_\epsilon\eta
  =
  \left(\sum_{j=0}^{c-2}\langle-1\rangle^j\right)\eta
  =
  \chi_2(c-1)\eta
  \in\KMW_{-1}(k).
\]
Thus this component is zero for \(c\) odd and is \(\eta\) for \(c\)
even.  This is the Milnor--Witt analogue of the parity appearing in
Rains' calculation, and it is the reason the final complex is an
\(\eta\)-twist of a derived orientation complex.

The main result identifies the cellular chain complex with that
nested-set model.  Let \(\Dor\) denote the derived orientation
coefficient complex and let \(C\mathcal N^+(\G,\Dor)\{\eta\}\) be the
\(\eta\)-twisted, strictly positive nested-set complex introduced
below.  Theorem \ref{nested-cellular} gives quasi-isomorphisms
\[
  C\mathcal N^+(\G,\Dor)\{\eta\}
  \xrightarrow{\ \simeq\ }
  \widetilde C_*^{\mathrm{cell}}(\PP(\G)),
  \qquad
  \mathbf Z\oplus C\mathcal N^+(\G,\Dor)\{\eta\}
  \xrightarrow{\ \simeq\ }
  C_*^{\mathrm{cell}}(\PP(\G)).
\]
Equivalently, the \(A\)-graded summand is computed by the nested sets
with root \(A\).  Taking the image of multiplication by \(\eta\) gives the
explicit formula
\[
  \eta\mathbf H_i^{\mathrm{cell}}(\PP(\G))
  \cong
  \bigoplus_{A\in\mathcal C_{\G}^{(2)}}
  H^{\dim A-i}
  \bigl([\hat0,A]_{\mathcal C_{\G}^{(2)}};\mathbf I^i\bigr),
  \qquad i>0,
\]
and after inverting \(\eta\),
\[
  \mathbf H_i^{\mathrm{cell}}(\PP(\G))[\eta^{-1}]
  \cong
  \bigoplus_{A\in\mathcal C_{\G}^{(2)}}
  H^{\dim A-i}
  \bigl([\hat0,A]_{\mathcal C_{\G}^{(2)}};\mathbf W\bigr),
  \qquad i>0.
\]
Here \(\mathcal C_{\G}^{(2)}\) is the subposet generated by direct sums
of even-dimensional members of the building set.  Cohomology with a
coefficient sheaf means that the integral interval complex is tensored
with that sheaf before taking cohomology.  When the integral interval
cohomology is free, these formulas reduce to tensor products of integral
cohomology with \(\mathbf I^i\) and \(\mathbf W\), respectively.
The image and localization conventions are specified in Section~7.
The $\eta$-twist is applied only to the strictly positive reduced
complex.  The independent augmentation summand gives
$\mathbf H_0^{\mathrm{cell}}(\PP(\G))=\mathbf Z$.

Over $\R$, the same complex gives an integral chain-level model in
which the two possible Milnor--Witt attaching classes become $0$ and
multiplication by $2$.  Rains' description of
$2H_*(\PP(\G)(\R);\Z)$ is recovered by passing to homology and
multiplying by $2$, equivalently by quotienting by the subgroup
annihilated by $2$.  The chain-level model retains the integral
attaching maps.

As an application, we specialize to the braid arrangement and hence to
the moduli spaces \(\overline{\mathcal M}_{0,N}\).  In this case the
\(2\)-divisible condition becomes the elementary condition that every
block of the corresponding partition is odd.  The cellular
\(\A^1\)-homology therefore splits into odd-partition summands together
with explicit two-term \(\eta\)-cones.  We work out the resulting
decompositions for \(\overline{\mathcal M}_{0,5}\) and
\(\overline{\mathcal M}_{0,6}\), illustrating how the motivic complex
remembers both the integral real homology detected after multiplying
by \(\eta\) and the remaining \(\eta\)-primary information.
Over the real numbers we also extract the resulting Chow--Witt groups
and the full untwisted Chow--Witt ring from the integral cohomology ring
of the real locus.

A recent complementary approach of Hennig starts directly with the
boundary stratification of $\overline{\mathcal M}_{0,N}$ and computes
cellular complexes coefficient by coefficient in Rost--Schmid notation
\cite[Section~3.2.3]{Hennig2026CellularM0n}.  Its strata are products of
spaces $M_{0,a}$ and hence products of hyperplane-arrangement
complements.  Hennig computes the Milnor--Witt cohomology of
$\overline{\mathcal M}_{0,5}$ and conjectures a three-part decomposition
for arbitrary $N$ and arbitrary line-bundle twists
\cite[Conjecture~3.2.2]{Hennig2026CellularM0n}.  Corollary
\ref{mw-cohomology-moduli} proves the trivial-line-bundle case for every
$N$: the free, $\eta$-kernel, and $\eta$-cokernel multiplicities are the
explicit integers $b_i$, $r_{i+1}$, and $r_i$.

We also construct the line-bundle-twisted forest residue complex
and prove its comparison with Nisnevich cohomology in Theorem
\ref{twisted-cellular-comparison}.  For the braid arrangement,
Proposition \ref{hennig-local-system-criterion} gives an integer
$\eta$-model for every line-bundle twist.  Corollary \ref{hennig-proved-cases} treats
the trivial and root-orientation twists and their canonical duals in every arity,
and all line-bundle twists for $N\leq6$, with the six-point case
verified by an exact computer-assisted calculation.  A blow-up proof of the
general case would require saturated images of the global Gysin and
restriction maps, which the local attaching-class calculation alone
does not establish.

The paper is organized as follows.  We first compare the nested-set
construction with the toric \((K,\lambda)\)-model, then review building
sets, wonderful models, and the geometric-context formalism, including
the idempotent decomposition of the cellular theory.  We next introduce
the derived nested-set complexes and prove their behavior under a
single blow-up.  The following section recalls the relevant
\(\A^1\)-homotopy blow-up triangle and computes its Milnor--Witt
connecting class.  Finally, we combine these ingredients to prove the
cellular quasi-isomorphism and derive the explicit formulas above,
followed by examples.  The final section specializes to the braid
arrangement, determines its untwisted Milnor--Witt cohomology and
Chow--Witt ring, and concludes with the twisted forest complex and
the stated twisted cases of Hennig's conjecture.

\subsection*{Notations}

Throughout, $k$ is a perfect field of characteristic different from
$2$.  Let $Sm_k$ be the category of smooth $k$-schemes of finite type,
equipped with the Nisnevich topology, and let
$\Delta^{op}Shv_{Nis}(Sm_k)$ be the category of simplicial Nisnevich
sheaves.  We refer to its objects as spaces.  The Nisnevich local
injective model structure has homotopy category $\mathcal H_s(k)$.
Left Bousfield localization at the projections
$\mathcal X\times\A^1\to\mathcal X$ gives the $\A^1$-model structure,
whose homotopy category is denoted by $\mathcal H(k)$.  Starting with
pointed spaces gives the pointed category $\mathcal H_\bullet(k)$.

We write
\[
  \KMW_*(k)=\bigoplus_{n\in\Z}\KMW_n(k)
\]
for Milnor--Witt K-theory.  It is generated by $\eta$ in degree $-1$
and by $[u]$ in degree $1$, for $u\in k^\times$, subject to
\[
  [u][1-u]=0,\qquad
  [uv]=[u]+[v]+\eta[u][v],\qquad
  \eta[u]=[u]\eta,\qquad
  \eta(2+\eta[-1])=0.
\]
We put $\langle u\rangle=1+\eta[u]\in\KMW_0(k)=\operatorname{GW}(k)$.
The unramified Milnor and Milnor--Witt K-theory sheaves are denoted by
$\mathrm K_n^{\mathrm M}$ and $\KMW_n$, respectively; both are strictly
$\A^1$-invariant \cite{Mor10}.

\subsection*{Cellular conventions}

We recall only the part of cellular $\A^1$-homology needed below.
Write $Ab_{\A^1}(k)$ for the abelian category of strictly
$\A^1$-invariant Nisnevich sheaves.  A smooth $k$-scheme $Y$ is
\emph{cohomologically trivial} if
\[
  H^q_{\mathrm{Nis}}(Y,M)=0
  \qquad(q>0)
\]
for every $M\in Ab_{\A^1}(k)$.  A cellular structure on a smooth
$k$-scheme $X$ is an increasing filtration by open subschemes
\[
  \varnothing=\Omega_{-1}X
  \subset\Omega_0X\subset\cdots\subset\Omega_sX=X
\]
such that
\[
  X_i:=\Omega_iX\setminus\Omega_{i-1}X
  =\coprod_{\alpha\in J_i}Y_{i\alpha}
\]
is smooth of pure codimension $i$ in $\Omega_iX$ and every
$Y_{i\alpha}$ is affine and cohomologically trivial.

Let $\nu_i$ be the normal bundle of $X_i\hookrightarrow\Omega_iX$.
Homotopy purity and the exact couple of the filtration give connecting
morphisms
\[
  \partial_i:
  \widetilde{\mathbf H}^{\A^1}_i(\operatorname{Th}(\nu_i))
  \longrightarrow
  \widetilde{\mathbf H}^{\A^1}_{i-1}
  (\operatorname{Th}(\nu_{i-1})).
\]
They form the cellular $\A^1$-chain complex
\[
  C_i^{\mathrm{cell}}(X)
  :=
  \widetilde{\mathbf H}^{\A^1}_i(\operatorname{Th}(\nu_i)).
\]
After choosing orientations of the normal bundles, the Thom
isomorphisms identify its terms with
\begin{equation}
  C_i^{\mathrm{cell}}(X)
  \cong
  \bigoplus_{\alpha\in J_i}
  \KMW_i\otimes\mathbf Z_{\A^1}[Y_{i\alpha}],
  \qquad i>0.
  \label{eq:oriented-cellular-terms}
\end{equation}
In degree zero one has
\[
  C_0^{\mathrm{cell}}(X)=\mathbf Z_{\A^1}[X_0].
\]
The index $i$ is both the codimension of the stratum and the
homological degree of its Thom term; this is the degree convention used
for every shift below.  Changes of Thom orientation act through
Milnor--Witt degrees.  In the nested-set model these choices and their
compatibilities are recorded by the derived orientation coefficient
complex $\Dor$.

For the wonderful models considered here, we use the model constructed
in the proof of Corollary~\ref{wonderful-blowup}, which gives a chain
homotopy equivalence
\begin{equation}
  C_*^{\mathrm{cell}}(X)
  \simeq
  \mathbf Z\oplus\widetilde C_*^{\mathrm{cell}}(X).
  \label{eq:cellular-augmentation-splitting}
\end{equation}
The first summand is concentrated in degree zero, and the reduced
summand on the right is strictly positive, with each degree-$i$ term
a finite sum of $\KMW_i$.  All $\eta$-twists used below refer to this
reduced model.

Finally, if $M\in Ab_{\A^1}(k)$, cellular cochains compute Nisnevich
cohomology:
\begin{equation}
\begin{split}
  H^q_{\mathrm{Nis}}(X,M)
  &\cong
  H^q\operatorname{Hom}_{Ab_{\A^1}(k)}
  \bigl(C_*^{\mathrm{cell}}(X),M\bigr)\\
  &\cong
  \operatorname{Hom}_{D(Ab_{\A^1}(k))}
  \bigl(C_*^{\mathrm{cell}}(X),M[q]\bigr).
\end{split}
\label{eq:cellular-cohomology-comparison}
\end{equation}
This is the comparison of
\cite[Proposition 2.27]{MOREL2023109346}; taking $M=\KMW_q$ is the
passage from the cellular calculations to Chow--Witt groups used in
Section~8.

Hennig uses the equivalent coefficientwise complex in homological
Rost--Schmid notation
\cite[Definition~2.2.2 and Theorem~2.2.3]{Hennig2026CellularM0n}.
After the Poincar\'e-duality shift, it is obtained by evaluating the
universal complex above on the chosen coefficient sheaf.  We retain the
universal formulation because a single chain-level decomposition then
computes all strictly $\A^1$-invariant coefficients simultaneously.

\subsection*{Acknowledgments}
The authors would like to thank Jeremy Usatine, Kirsten Wickelgren and Nanjun Yang for their helpful discussions and advice.
%suggestions 

%% file: combinatorial.tex
\section{Subspace arrangements and De Concini--Procesi models}

The toric calculation of
\cite{liu2025cellularmathbba1homologysmoothtoric} provides a useful
prototype for the construction below.  A smooth fan may be encoded by a
pair $\Sigma=(K,\lambda)$, where $K$ is the simplicial complex of cones
and
\[
  \lambda:\Z^{K(0)}\longrightarrow N
\]
sends the standard basis vectors to the primitive ray generators.  Assume
here that the rays span $N\otimes_{\Z}\R$, so that there is no torus
factor.  If
\[
  \AZ_K=(\A^1,\Gm)^K\subseteq\A^{K(0)}
\]
is the motivic moment-angle space, the Cox quotient gives
\[
  X_\Sigma\cong\AZ_K/\ker(\exp\lambda).
\]
The canonical cubical cells of $\AZ_K$, together with the changes of
orientation induced by $\ker(\exp\lambda)$, produce the cellular
$\A^1$-chain complex of $X_\Sigma$
\cite[Propositions~2.23 and~3.4]{liu2025cellularmathbba1homologysmoothtoric}.
For a pure shellable fan, this complex retracts onto a combinatorial
complex built from the subcomplexes $K_\omega$, with
$\omega\in\operatorname{row}(\lambda\bmod 2)$; its elementary
Milnor--Witt extensions are cones of multiplication by $l\eta$
\cite[Corollary~3.11]{liu2025cellularmathbba1homologysmoothtoric}.

The common ground between that construction and the present one is
provided by Coxeter toric varieties.  Let $R$ be an irreducible root
system in a finite-dimensional Euclidean space.  The closures of its
Weyl chambers form a regular complete fan $\Sigma_R$ \cite{procesi}.
Write $X_R$ for the associated smooth projective toric variety, $K_R$
for its Coxeter complex, and $\lambda_R$ for its characteristic map.
For $R=A_n$, the vertices of $K_{A_n}$ are indexed by the nonempty
proper subsets of $\llBra{n+1}$, its simplices are chains of such
subsets, and, with $e_{n+1}=-\sum_{i=1}^n e_i$,
\[
  \lambda_{A_n}(L)=\sum_{i\in L}e_i.
\]
The variety $X_{A_n}$ is the permutohedral toric variety.  It is also
the De Concini--Procesi model obtained by blowing up the coordinate
linear strata of $\PP^n$ \cite{CP,Denham_2013}; hence its cellular
complex can be described both by the toric pair
$(K_{A_n},\lambda_{A_n})$ and by a building set.

For a general wonderful model there is no toric quotient.  The
simplicial complex $K$ is replaced by the nested-set complex, and the
orientation information carried by $\lambda$ is replaced by the
derived orientation coefficients $\Dor$.  The blow-up calculation in
Section~\ref{sec:blow-up} gives the local top cellular attaching class
$(c-1)_\epsilon\eta$ on an oriented normal slice of codimension $c$.  Thus the
toric $l\eta$-extensions and the $\eta$-twisted nested-set complex are
two instances of the same orientation-sensitive cellular mechanism.
The type-$A$ overlap above concerns the coordinate wonderful model; it
should not be confused with the braid-arrangement model in
Example~\ref{moduli}.

We now follow the dual-space convention of \cite{Rains_2010}.  A
\emph{subspace arrangement} in a finite-dimensional vector space $V$ is
a finite collection $\G$ of subspaces of $V^*$.  Let
\[
  \mathcal C_{\G}
  :=
  \left\{
    \sum_{G\in S}G : S\subseteq\G
  \right\},
\]
where the empty sum is $0$.  We will also use the corresponding
collections of actual subspaces and quotient spaces
\[
  \mathcal S_{\G}:=\{A^\perp\subseteq V:A\in\mathcal C_{\G}\},
  \qquad
  \mathcal Q_{\G}:=\{V/A^\perp:A\in\mathcal C_{\G}\}.
\]

For $U\in\mathcal C_{\G}$, a \emph{$\G$-decomposition} of $U$ is a
direct-sum decomposition
\[
  U=\bigoplus_{j=1}^r U_j,
  \qquad
  0\neq U_j\in\mathcal C_{\G},
\]
such that every $G\in\G$ contained in $U$ is contained in exactly one
$U_j$.  Let $\overline{\G}$ be the collection of
$\G$-indecomposable elements of $\mathcal C_{\G}$.  Then
$\mathcal C_{\G}=\mathcal C_{\overline{\G}}$, and a
\emph{building set} is an arrangement satisfying
$\G=\overline{\G}$.

For a linear map $f:V\to V'$ such that $f^*(G)\neq0$ for every
$G\in\G'$, put
\[
  f^{-1}\G'
  :=
  \overline{\{f^*(G):G\in\G'\}}.
\]
A morphism $\G\to\G'$ is a linear map $f:V\to V'$ such that
$f^*(G)\in\G$ for every $G\in\G'$.  These form the category
$\Build_k$, with the forgetful functor
\[
  P:\Build_k\longrightarrow\mathbf{Vect}_k.
\]
A \emph{weak morphism} $\G\dashrightarrow\G'$ is a linear map satisfying
$f^*(G)\in\G\cup\{0\}$ for every $G\in\G'$.

For a fixed vector space $V$, the fiber $\Build_k(V)$ is ordered by
reverse inclusion: $\operatorname{Hom}_{\Build_k(V)}(\G_1,\G_2)$ is
nonempty precisely when $\G_1\supseteq\G_2$.  For an arbitrary linear
map $f:V\to V'$, we also put
\[
  f^!\G'
  :=
  \overline{\{f^*(G):G\in\G',\ f^*(G)\neq0\}}.
\]

We use separate notation for operations on subspaces of $V$ and on
subspaces of $V^*$.  If $i_W:W\hookrightarrow V$ and
$q_W:V\twoheadrightarrow V/W$, define
\[
  i_W^!\G
  :=
  \overline{
    \{i_W^*(G):G\in\G,\ i_W^*(G)\neq0\}
  },
\]
\[
  q_{W,\#}\G
  :=
  \overline{\{G\in\G:G\subseteq W^\perp\}}.
\]
For $A\in\mathcal C_{\G}\subseteq V^*$, the notation of
\cite{Rains_2010} is
\[
  \G|_A
  :=
  q_{A^\perp,\#}\G
  =
  \{G\in\G:G\subseteq A\},
\]
\[
  \G/A
  :=
  i_{A^\perp}^!\G
  =
  \overline{
    \{(G+A)/A:G\in\G,\ G\nsubseteq A\}
  }.
\]
Thus $\G|_A$ is a building set on $V/A^\perp$, whereas $\G/A$
is a building set on $A^\perp$.  This convention removes the ambiguity
between restriction in $V$ and restriction in $V^*$.
When the actual subspace $W\subseteq V$ is the more natural label, we
write $\G|_W:=i_W^!\G$ and $\G/W:=q_{W,\#}\G$.  Thus, for
$W=A^\perp$, the actual-space notation gives $\G/W=\G|_A$ and
$\G|_W=\G/A$.  On the corresponding posets, $q_{W,\#}$ is left adjoint to
$q_W^{-1}$.  If
$\G\setminus W:=\{G\in\G:G\nsubseteq W^\perp\}$, then
$i_W^!\G=i_W^{-1}(\G\setminus W)$.

De Concini and Procesi associate a smooth projective variety to $\G$.
Set
\[
  \mathcal A_{\G}
  :=
  V\setminus\bigcup_{G\in\G}G^\perp,
  \qquad
  U_{\G}:=(\mathcal A_{\G}\setminus\{0\})/\Gm\subseteq\PP(V).
\]
For $G\in\G$, let $\PP_G:=\PP(V/G^\perp)$.  The diagonal map
\[
  \mathcal A_{\G}\longrightarrow\prod_{G\in\G}\PP_G
\]
factors through $U_{\G}$ when $\G\neq\varnothing$.  We denote the
closure of its image by
\[
  \PP(\G):=\overline{Y}_{\G}
\]
and write $\rho_G:\PP(\G)\to\PP_G$ for the natural projection.
If $\dim G=1$, equivalently if $G^\perp$ is a hyperplane, then
$\PP_G$ is a point.  Consequently, adjoining or removing such an
element has no effect on the model, provided the result remains a
building set.  The \emph{root} of $\G$ is the maximal element
$\operatorname{root}(\G)$ of $\mathcal C_{\G}$.

\begin{example}\label{moduli}
Let $m\geq3$, let $V=k^m/k\mathbf1$, and identify
\[
  V^*=\left\{(a_1,\ldots,a_m)\in(k^m)^*:
  \sum_i a_i=0\right\}.
\]
For $S\subseteq[m]$ with $|S|\geq2$, put
\[
  G_S:=\operatorname{span}\{e_i^*-e_j^*:i,j\in S\}\subseteq V^*.
\]
The building set of indecomposable flats of the braid arrangement is
\[
  \overline{A}_{m-1}:=\{G_S:S\subseteq[m],\ |S|\geq2\}.
\]
Since $\dim G_S=|S|-1$, the wonderful model embeds as
\[
  \PP(\overline{A}_{m-1})
  \hookrightarrow
  \prod_{\substack{S\subseteq[m]\\ |S|\geq2}}
  \PP(V/G_S^\perp)
  \cong
  \prod_{\substack{S\subseteq[m]\\ |S|\geq2}}
  \PP^{|S|-2}.
\]
The construction in \cite[Example following Theorem~2.5]{Rains_2010} identifies this
model canonically with the moduli space of stable pointed rational
curves:
\[
  \PP(\overline{A}_{m-1})
  \cong
  \overline{\mathcal M}_{0,m+1}.
\]
Under this isomorphism, the factor indexed by $S$ is obtained by
forgetting the markings outside $S\cup\{\infty\}$, stabilizing, and
contracting all components not containing $\infty$.  Place $\infty$ at
infinity on the resulting projective line and take the configuration
of the remaining markings modulo translations and scalings.
\end{example}

\section{Geometric models and section spaces}

We next reformulate wonderful models as section spaces of a diagram.
Let
\[
  \mathrm{Corr}(\mathrm{Var}_k)
\]
be the $2$-category whose objects are $k$-varieties, whose morphisms are spans
$X\leftarrow Z\rightarrow Y$, and whose $2$-morphisms are morphisms of
spans.

\begin{definition}
A \emph{geometric context} is a lax $2$-functor
\[
  \cF:\Vect_k\longrightarrow\mathrm{Corr}(\mathrm{Var}_k)
\]
such that
\begin{enumerate}
  \item $\cF(\mathbf 0)=\varnothing$, and $\cF(V)\neq\varnothing$ if
  $V\neq\mathbf 0$;
  \item for every linear map $f:V\to W$, the span
  $\cF(V)\leftarrow\cF(f)\rightarrow\cF(W)$ identifies $\cF(f)$ with a
  subvariety of $\cF(V)\times\cF(W)$, nonempty when $V,W\neq\mathbf 0$;
  \item if $i:V\hookrightarrow W$ is injective, then
  $\cF(i)$ is the span
  $\cF(V)\xleftarrow{=}\cF(V)\hookrightarrow\cF(W)$, and the analogous
  square for a pullback diagram of injections is cartesian;
  \item for a zero morphism $f_0:V\to\mathbf 0\to W$, the span is
  $\cF(V)\xleftarrow{\mathrm{pr}_1}\cF(V)\times\cF(W)
  \xrightarrow{\mathrm{pr}_2}\cF(W)$.
\end{enumerate}
\end{definition}

\begin{example}
The main example is the \emph{projective context}.  It sends a vector
space $V$ to $\PP(V)$, and a linear map $f:V\to W$ to the
correspondence
\[
  \PP(f)
  :=
  \{([v],[w])\in\PP(V)\times\PP(W): f(v)=\lambda w
  \text{ for some }\lambda\in k\}.
\]
\end{example}

\begin{example}
A related weak context is obtained by taking $U(V)=\PP(V)$ and
\[
  U(f):=
  \{([v],[w])\in\PP(V)\times\PP(W): f(v)=w
  \text{ after choosing representatives}\}.
\]
For a building set containing $V^*$, this keeps track of the complement
$U_{\G}$ rather than its projective closure.
\end{example}

\begin{remark}
There is also a sphere variant, useful for comparing with sphere
blow-ups:
\[
  \bS(V):=\afnz{V},
  \qquad
  \bS(f):=
  \{(v,w,\lambda):f(v)=\lambda w\}.
\]
This weakens the monomorphism condition in the definition above.
\end{remark}

The incidence category has a projection
$p:\int\cF\to\Vect_k$.  Its objects are pairs $(V,x)$ with
$x\in\cF(V)$, and
\[
  \operatorname{Hom}_{\int\cF}((V,x),(W,y))
  =
  \{f:V\to W:(x,y)\in\cF(f)\}.
\]
For a finite category $I$ and a diagram $D:I\to\Vect_k$, define the
$\cF$-model of $D$ to be the space of sections
\[
  \cF(D):=\Gamma(D,\cF)
  =
  \operatorname{Fun}_D(I,\int\cF).
\]
Equivalently, if $D_0$ and $D_1$ denote the objects and morphisms of
the diagram, then
\[
  \cF(D)
  =
  \{(x_a)\in\prod_{a\in D_0}\cF(D(a)):
  (x_{s(\alpha)},x_{t(\alpha)})\in\cF(D(\alpha))
  \text{ for every }\alpha\in D_1\}.
\]

A building set $\G\in\Build_k(V)$ is viewed as the diagram
\[
  |\G|\longrightarrow\Vect_k,
  \qquad
  G\longmapsto V/G^\perp .
\]
For the projective context this construction recovers the
De Concini--Procesi model:
\[
  \PP(\G)=\Gamma(\G,\PP)=\overline Y_{\G}.
\]
If $V^*\in\G$, the complement context recovers $U_{\G}$.

For the projective context, functoriality is given directly on the
factors.  A building-set morphism $f:\G_1\to\G_2$, with underlying
linear map $f:V_1\to V_2$, induces, for each $G\in\G_2$, the map
\[
  \PP(\G_1)\xrightarrow{\rho_{f^*G}}
  \PP\bigl(V_1/(f^*G)^\perp\bigr)
  \hookrightarrow\PP(V_2/G^\perp).
\]
These maps satisfy the incidence relations and define
$\PP(f):\PP(\G_1)\to\PP(\G_2)$
\cite[Proposition~2.4]{Rains_2010}.  For a weak morphism, the factors
killed by $f^*$ are supplied by a second wonderful model, giving
\[
  \PP(\G_1)\times\PP(\G_2/\operatorname{Im}(f))
  \longrightarrow\PP(\G_2).
\]
For $f^*G\neq0$ use the displayed factor map, and for $f^*G=0$ use
the $G$-coordinate of the second factor.  These are the operad maps
of \cite[Theorem~2.5]{Rains_2010}, with its composition law.

\section{Decomposition}

We now use the projective context.  For $A\in\mathcal C_{\G}$, the
subdiagram $\G|_A\subseteq\G$ gives a natural projection
\[
  \pi_A:\PP(\G)\longrightarrow\PP(\G|_A).
\]
The weak morphism associated with the inclusion $A^\perp\hookrightarrow V$
gives an operad map
\[
  \PP(\G/A)\times\PP(\G|_A)\longrightarrow\PP(\G).
\]
Choosing a $k$-point of $\PP(\G/A)$ gives a section
$\pi_A^{[-1]}$ \cite[proof of Corollary~2.8]{Rains_2010}.  Such points
exist because wonderful models are obtained from products of projective
spaces by blow-ups along smooth centers.  The induced retractions
$e_A=\pi_A^{[-1]}\pi_A$ on cellular complexes satisfy
\[
  e_Ae_B=e_{A\wedge B}
  \quad\text{in }D(Ab_{\A^1}(k)),
\]
where $A\wedge B$ is the meet in $\mathcal C_{\G}$.  The independence
of the parameter points and this identity are proved in the proof of
Corollary~\ref{wonderful-blowup}, using the operad composition law.
Splitting these commuting idempotents gives summands
$\widetilde C_*^{\mathrm{cell}}(\PP(\G))[V/A^\perp]$ in
$D(Ab_{\A^1}(k))$ and hence the grading
\[
  \widetilde{\mathbf H}^{\mathrm{cell}}_*(\PP(\G))
  =
  \bigoplus_{A\in\mathcal C_{\G}}
  \widetilde{\mathbf H}^{\mathrm{cell}}_*(\PP(\G))[V/A^\perp],
\]
where the $V/A^\perp$-summand is fixed by
$\pi_A^{[-1]}\pi_A$ and annihilated by
$\pi_B^{[-1]}\pi_B$ for every $B\subsetneq A$.  We also write this
summand as $[A]$ when using the poset label.  In particular, the
$V^*$-graded root summand is denoted by $[V]$, in agreement with the
geometric notation of the preceding sections.

\section{Nested sets of building sets and derived twisted complexes}

For $m\geq1$, let $\mathcal C_{\G}^{(m)}$ be the subposet of
$\mathcal C_{\G}$ consisting of elements that are direct sums of
members $G\in\G$ with $\dim G$ divisible by $m$.

\begin{definition}
A \emph{nested set} is a subset $F\subseteq\G$ such that every
collection of pairwise incomparable elements of $F$ is a
$\G$-decomposition of its sum.  Following \cite{Rains_2010}, such a
set may equivalently be called a $\G$-forest.  Its root is
\[
  \operatorname{root}(F):=\sum_{G\in F}G\in\mathcal C_{\G},
\]
and, for $G\in F$, its child is
\[
  \operatorname{child}_F(G)
  :=
  \sum_{\substack{H\in F\\H\subsetneq G}}H.
\]
The nested set $F$ is \emph{$m$-divisible} if $m$ divides
$\dim G$ for every $G\in F$.
\end{definition}

For $A\in\mathcal C_{\G}^{(m)}$, let
$\mathcal N_A^{(m)}(\G)$ be the poset of $m$-divisible nested sets
with root $A$.  The complex
$\mathcal{CN}^{(m)}_*(A;R)$ is freely generated in degree $r$ by
ordered members of $\mathcal N_A^{(m)}(\G)$ with $r$ nodes, with
different orderings identified by the usual sign.  If $G=G_i\in F$,
set
\[
  \partial_G F=(-1)^{i-1}(F\setminus\{G\}),
\]
and set $\partial_G F=0$ if $G\notin F$.  The differential is the sum
over the elements $G$ that are proper subspaces of the components of
$A$.  Rains' forest subdivision gives a canonical isomorphism
\cite[Theorem 3.2]{Rains_2010}
\[
  H_*\bigl(\mathcal{CN}^{(m)}_*(A;R)\bigr)
  \cong
  H_*\bigl([\hat0,A]_{\mathcal C_{\G}^{(m)}};R\bigr).
\]
Here interval homology has Rains' relative convention: for $A\neq\hat0$,
\[
  H_r([\hat0,A];R)=\widetilde H_{r-2}((\hat0,A);R).
\]
The endpoints are fixed in the relative chain complex; this is not the
ordinary homology of the closed interval.  For $A=\hat0$ the interval
complex is $R$ in degree zero.  We use the dual convention for interval
cohomology.
For $A\notin\mathcal C_{\G}^{(m)}$, the corresponding interval
complex is understood to be zero.

For a nonzero vector space $W$, write
\[
  \operatorname{or}(W)
  :=
  \widetilde{\mathbf H}^{\A^1}_{\dim W}
  \bigl(W/(W\setminus\{0\})\bigr)
  \cong
  \KMW_{\dim W}.
\]
Set $\operatorname{or}(0)=\Z$.  With tensor products taken in
$Ab_{\A^1}(k)$, every short exact sequence
$0\to W'\to W\to W''\to0$ induces
\[
  \operatorname{or}(W')\otimes\operatorname{or}(W'')
  \xrightarrow{\ \cong\ }
  \operatorname{or}(W).
\]

We use the derived-orientation notation of the preceding sections.
Let $\chi_2(r)\in\{0,1\}$ denote the parity of $r$, and define
\[
  \operatorname{Dor}(W)_i
  =
  \begin{cases}
    \Z e(W)_i,&0\leq-i<\dim W-1,\\
    0,&\text{otherwise},
  \end{cases}
  \qquad
  \delta_i=\chi_2(\dim W+i).
\]
Here $\delta_i$ sends $e(W)_i$ to
$\chi_2(\dim W+i)e(W)_{i-1}$; generators outside the displayed
range are zero.  We use the homological shift convention
$(C[r])_j=C_{j-r}$, with differential $(-1)^r d_C$.
For $W\neq0$,
\[
  \operatorname{Dor}(W)
  \simeq
  \begin{cases}
    \Z,&\dim W\ \text{even},\\
    0,&\dim W\ \text{odd},
  \end{cases}
  \qquad
  \widehat{\operatorname{Dor}}(W)
  :=
  \operatorname{Dor}(W)[\dim W-1].
\]
The definition gives $\operatorname{Dor}(0)=0$.
For a subspace $H\subseteq W$, the short exact sequence
$0\to H\to W\to W/H\to0$ induces the structure map
$D_H:\operatorname{Dor}(W)\to
\operatorname{Dor}(H)\otimes\operatorname{Dor}(W/H)$, given on
generators by
\[
  D_H(e(W)_i)
  =
  \begin{cases}
    e(H)_i\otimes e(W/H)_0,
    &0\leq-i<\dim H-1\text{ and }\chi_2(\dim W/H)=0,\\
    0,&\text{otherwise}.
  \end{cases}
\]
This is a chain map: when $\dim(W/H)$ is even, the two internal
differentials have the same parity, and the truncation is a quotient
of complexes; when it is odd, the map is zero.
For a nested set $F$, the coefficient complex is
\[
  \operatorname{Dor}(F)
  :=
  \bigotimes_{G\in F}
  \operatorname{Dor}
  \bigl(\operatorname{child}_F(G)^\perp/G^\perp\bigr).
\]
The structure maps above define a forest cochain complex
$\mathcal{CN}_A(\G,\operatorname{Dor})$ as follows, and we put
\[
  \mathcal{CN}(\G,\operatorname{Dor})
  :=
  \bigoplus_{A\in\mathcal C_{\G}}
  \mathcal{CN}_A(\G,\operatorname{Dor}).
\]
We distinguish the cohomological grading of
$\mathcal{CN}_A(\G,\operatorname{Dor})$ from its cellular chain
grading.  Its generators, together with an ordering of $F$, are
\[
  \bigotimes_{G\in F}
  e\bigl(\operatorname{child}_F(G)^\perp/G^\perp\bigr)_{i_G},
  \qquad
  |F|-\sum_{G\in F}i_G=n,
\]
in cochain degree $n$.  Put $p=|F|$ and $q=-\sum i_G$, so that $n=p+q$.
The internal differential $\delta$ uses the tensor-product signs.
The map $D$ adds a node in position $a$ with sign $(-1)^{a-1}$ and
applies the corresponding structure map to its parent coefficient;
it is zero if the enlarged set is not nested.  Reordering a forest
uses its orientation sign and the Koszul sign of its coefficient
tensor.  The differential is $d=\delta+(-1)^nD$.
The structure maps commute under successive refinements: both
composites vanish unless the two new residual dimensions are even,
and otherwise both retain the old indices and insert the two
degree-zero generators.  Thus $D^2=0$ and $\delta D=D\delta$, which
give $d^2=0$.

For $A=W^\perp$, the chain complex used below is
\[
 C_j\mathcal N(V/W,\G/W,\operatorname{Dor})
 :=\mathcal{CN}_A^{\dim A-j}(\G,\operatorname{Dor}),
\]
with the same differential.  All tensor products below are tensor
products of these chain complexes.  Explicitly, concatenating two
forests of bidegrees $(p_1,q_1)$ and $(p_2,q_2)$, whose roots have
dimensions $d_1$ and $d_2$, respectively, uses the sign
$(-1)^{(d_1+p_1)(p_2+q_2)}$.  Comparing the two internal and the two
node-insertion differentials verifies this tensor-product
identification.  For a singleton forest, the identification with
$\widehat{\operatorname{Dor}}(A^*)$ sends $e(A^*)_i$ to
$(-1)^{(\dim A-1)i}e(A^*)_i$.
Its cohomology is the cohomology of the $2$-divisible interval:
\[
  H^i\bigl(\mathcal{CN}_A(\G,\operatorname{Dor})\bigr)
  \cong
  H^i\bigl([\hat0,A]_{\mathcal C_{\G}^{(2)}};\Z\bigr).
\]

The nested complex behaves well under the blow-up operation on building
sets.  Let $\G'$ be a building set with root $V^*$, let
$G\subsetneq V^*$ satisfy
$G\notin\mathcal C_{\G'}$, and suppose
$\G=\G'\sqcup\{G\}$ is again a building set.  Put
\[
  B_G
  :=
  C\mathcal N(V/G^\perp,\G|_G,\operatorname{Dor})
  \otimes
  C\mathcal N(G^\perp,\G/G,\operatorname{Dor}).
\]
There is a degreewise split exact sequence of complexes
\begin{equation}
  0\to
  B_G
  \to
  C\mathcal N(V,\G,\operatorname{Dor})
  \to
  C\mathcal N(V,\G',\operatorname{Dor})
  \to0.
  \label{eq:absolute-nested-blowup}
\end{equation}
The projection discards the generators whose forests contain $G$.

\begin{lemma}\label{AbsBlUp}
The sequence \eqref{eq:absolute-nested-blowup} is exact.  Its
connecting morphism is the signed insertion of the node $G$:
\[
  \gamma_G:
  C\mathcal N(V,\G',\operatorname{Dor})
  \longrightarrow
  B_G[1].
\]
It is natural for isomorphisms of building sets preserving $G$.
\end{lemma}

\begin{proof}
Since $G\notin\mathcal C_{\G'}$, a collection of old nodes is nested
for $\G$ if and only if it is nested for $\G'$.  The generators
containing $G$ form a subcomplex, since the differential only changes
indices or inserts nodes.  Such a forest separates into its nodes
contained in $G$ and its contracted forest modulo $G$.  This is a
bijection with pairs of forests for $\G|_G$ and $\G/G$, with the
first forest rooted at $G$.  Indeed, a quotient node lifts uniquely
to the component other than $G$, or to the component containing
$G$, in the $\G$-decomposition of its inverse image.  Uniqueness of
these decompositions gives the inverse bijection and preserves the
nesting condition.  The residual quotient spaces at corresponding
nodes are identical.  With the tensor-product sign specified above,
restriction and contraction therefore identify this subcomplex
with $B_G$, proving exactness.

Write $A=C\mathcal N(V,\G',\operatorname{Dor})$.  Lifting each old
forest to itself gives a graded splitting of the sequence.  Relative
to $C\mathcal N(V,\G,\operatorname{Dor})=B_G\oplus A$, its
differential is
\[
 \begin{pmatrix}d_{B_G}&t_G\\0&d_A\end{pmatrix},
 \qquad d_{B_G}t_G+t_Gd_A=0.
\]
The off-diagonal entry consists exactly of insertion of $G$.
Consequently $t_G:A\to B_G[1]$ is the connecting chain map.  All
constructions commute with the indicated isomorphisms.
\end{proof}

Assume now that $G$ is minimal in $\G$.  Then
\[
  C\mathcal N(V/G^\perp,\G|_G,\operatorname{Dor})
  \cong
  \widehat{\operatorname{Dor}}(V/G^\perp).
\]
Put $c=\dim G$.  Call $F\in\mathcal N(V,\G')$ admissible if
$F\cup\{G\}$ is nested.  For such $F$, let $m=m_G(F)$ be its
unique minimal member containing $G$, and put
\[
 r_F=\dim m-\dim\operatorname{child}_F(m).
\]
Define $\mathcal U_{\setminus G}(V,\G',\operatorname{Dor})$ to be
the subgroup spanned by the nonadmissible forests and by the
admissible tensors with $-i_m>r_F-c-2$.
The corresponding Thom complex is
\[
  \operatorname{Th}_G(G^\perp,\G/G,\operatorname{Dor})
  :=
  C\mathcal N(V,\G',\operatorname{Dor})/
  \mathcal U_{\setminus G}(V,\G',\operatorname{Dor}).
\]
This subgroup is a subcomplex.  The internal differential preserves
the discarded range.  An insertion away from $m$ leaves the cutoff
unchanged, whereas an insertion below $m$ not containing $G$ can
only lower it.  If an insertion replaces $m$ as the minimal ancestor
of $G$, its residual dimension is at least $c$, so its structure map
already kills every discarded tensor.  Nonadmissibility is
preserved by insertion.
Restriction and contraction identify the underlying graded group
of this quotient with
$C\mathcal N(G^\perp,\G/G,\operatorname{Dor})[c]$.
Indeed, the residual dimension at $m$ decreases by $c$, and all
other residual dimensions and indices are unchanged.  If $c$ is
even, the parity of every structure map is unchanged, so this is
an isomorphism of complexes.  For odd $c$ it is only an
identification of graded groups.

\begin{lemma}\label{AbsGysin}
For $G$ minimal in $\G$, the connecting morphism factors as
\[
  C\mathcal N(V,\G',\operatorname{Dor})
  \longrightarrow
  \operatorname{Th}_G(G^\perp,\G/G,\operatorname{Dor})
  \xrightarrow{\ \chi_2(c-1)\,\iota_G\ } B_G[1].
\]
For even $c$, the map $\iota_G$ inserts the top generator
$e(V/G^\perp)_0$ of $\widehat{\operatorname{Dor}}(V/G^\perp)$.
For odd $c$, the second arrow is zero.
\end{lemma}

\begin{proof}
For an admissible forest, insertion of $G$ replaces the parent
residual space $W_F=\operatorname{child}_F(m)^\perp/m^\perp$ by
the short exact sequence
\[
 0\longrightarrow
 (G+\operatorname{child}_F(m))^\perp/m^\perp
 \longrightarrow W_F\longrightarrow V/G^\perp\longrightarrow0.
\]
The last identification uses
$G\cap\operatorname{child}_F(m)=0$, which follows from minimality
and admissibility.  The first term has dimension $r_F-c$.
The defining formula for $D_H$ therefore retains precisely the
indices $-i_m\leq r_F-c-2$, inserts $e(V/G^\perp)_0$, and is zero
unless $c$ is even.  It vanishes on every nonadmissible forest as
well.  Order $G$ first in the new forest.  Its incidence sign is
$1$, and the total differential contributes $(-1)^{p+q}$.
For even $c$, this is exactly the tensor-product identification
sign $(-1)^{(c+1)(p+q)}$ above.  Thus the induced map is
$\iota_G$ with the displayed shifts.  The top generator is a
cycle when $c$ is even, so $\iota_G$ is a chain map; when $c$ is
odd the connecting map is zero.  This proves the factorization.
\end{proof}

%% file: blow-up.tex
\section{Blow-ups and motivic homotopy}\label{sec:blow-up}

The construction of $\PP(\G)$ by iterated blow-ups is the geometric
input for the calculation of its cellular $\A^1$-homology.  We first
record the motivic blow-up triangle in notation compatible with the
preceding section.

Let $X$ be a smooth $k$-scheme, let $i:Z\hookrightarrow X$ be a smooth
closed subscheme of pure codimension $c$, and let
\[
  p:\widetilde X:=\operatorname{Bl}_Z X\longrightarrow X
\]
be the blow-up.  Its exceptional divisor is
$E=p^{-1}(Z)\cong\PP(N_{Z/X})$.

\subsection{Purity and the blow-up triangle}

The square
\[
\begin{tikzcd}
  E \arrow[r,hook] \arrow[d] &
  \widetilde X \arrow[d,"p"]\\
  Z \arrow[r,hook] & X
\end{tikzcd}
\]
gives the canonical blow-up distinguished triangle
\cite[Proposition 2.29 and Remark 2.30]{MV99}
\begin{equation}
  \widetilde C_*^{\A^1}(E)
  \longrightarrow
  \widetilde C_*^{\A^1}(Z)
  \oplus
  \widetilde C_*^{\A^1}(\widetilde X)
  \longrightarrow
  \widetilde C_*^{\A^1}(X)
  \longrightarrow
  \widetilde C_*^{\A^1}(E)[1].
  \label{eq:blowup-distinguished}
\end{equation}
\begin{proposition}\label{blowup}
There is a distinguished triangle in $D_{\A^1}(k)$
\begin{equation}
  \widetilde C_*^{\A^1}(\widetilde X)
  \longrightarrow
  \widetilde C_*^{\A^1}(X)
  \xrightarrow{\delta}
  \widetilde C_*^{\A^1}
  \bigl(\operatorname{cofib}(E_+\to Z_+)\bigr)
  \longrightarrow
  \widetilde C_*^{\A^1}(\widetilde X)[1].
  \label{eq:blowup-cofiber}
\end{equation}
Moreover, $\delta$ factors as
\begin{equation}
  \widetilde C_*^{\A^1}(X)
  \longrightarrow
  \widetilde C_*^{\A^1}
  \bigl(\operatorname{Th}(N_{Z/X})\bigr)
  \longrightarrow
  \widetilde C_*^{\A^1}
  \bigl(\operatorname{cofib}(E_+\to Z_+)\bigr).
  \label{eq:purity-factor}
\end{equation}

Suppose, in addition, that the four schemes in the blow-up square are
equipped with cellular filtrations preserved by all four maps, and that
the square at every filtration stage is again a blow-up square.  Put
\begin{equation}
  Q_*^{\mathrm{cell}}(E,Z)
  :=
  \operatorname{Cone}\left(
    \widetilde C_*^{\mathrm{cell}}(E)
    \longrightarrow
    \widetilde C_*^{\mathrm{cell}}(Z)
  \right).
  \label{eq:cellular-exceptional-cofiber}
\end{equation}
Then there is a distinguished triangle in $D(Ab_{\A^1}(k))$
\begin{equation}
  \widetilde C_*^{\mathrm{cell}}(\widetilde X)
  \longrightarrow
  \widetilde C_*^{\mathrm{cell}}(X)
  \xrightarrow{\delta_{\mathrm{cell}}}
  Q_*^{\mathrm{cell}}(E,Z)
  \longrightarrow
  \widetilde C_*^{\mathrm{cell}}(\widetilde X)[1].
  \label{eq:cellular-blowup-cofiber}
\end{equation}
With the filtrations induced on the quotient and the Thom space,
$\delta_{\mathrm{cell}}$ factors as
\begin{equation}
  \widetilde C_*^{\mathrm{cell}}(X)
  \longrightarrow
  \widetilde C_*^{\mathrm{cell}}
  \bigl(\operatorname{Th}(N_{Z/X})\bigr)
  \longrightarrow
  Q_*^{\mathrm{cell}}(E,Z).
  \label{eq:cellular-purity-factor}
\end{equation}
\end{proposition}

\begin{proof}
The triangle is the cofiber form of
\eqref{eq:blowup-distinguished}.  Since $p$ is an
isomorphism over $X\setminus Z$, the map from $X$ to the cofiber
factors through
\[
  X/(X\setminus Z)
  \simeq_{\A^1}
  \operatorname{Th}(N_{Z/X})
\]
by homotopy purity \cite[Theorem 2.23]{MV99}.

For the cellular assertion, the stagewise blow-up squares give a
filtered form of \eqref{eq:blowup-distinguished}.  Passing to filtration
quotients and applying reduced $\A^1$-homology in the corresponding
cellular degrees gives a morphism of exact couples.  Its $E^1$-complex
is the cellular chain complex, and its cofiber sequence is
\eqref{eq:cellular-blowup-cofiber}.  Homotopy purity is compatible with
the induced filtrations, so the same construction applied to the
factorization through $X/(X\setminus Z)$ gives
\eqref{eq:cellular-purity-factor}.
\end{proof}

\subsection{The Milnor--Witt connecting class}

For $r\geq0$, put
\[
  r_\epsilon
  :=
  \sum_{j=0}^{r-1}\langle-1\rangle^j
  \in\KMW_0(k),
  \qquad
  0_\epsilon:=0.
\]

\begin{corollary}\label{blowup-parity}
On a normal slice of codimension $c\geq2$, choose coordinates and the
induced standard cellular orientations.  The top cellular component of
the last arrow in \eqref{eq:cellular-purity-factor} is the map
$\KMW_c\to\KMW_{c-1}$ given by multiplication by
\begin{equation}
  (c-1)_\epsilon\eta=\chi_2(c-1)\eta.
  \label{eq:blowup-parity}
\end{equation}
Since $(1+\langle-1\rangle)\eta=0$, this is $0$ when $c$ is odd and
is $\eta$ when $c$ is even.  For $c=1$, the blow-up is an isomorphism
and the connecting morphism is zero.  Under complex Betti realization
with integral coefficients the local component is zero.  Under real Betti
realization it is zero in odd codimension and, in even codimension,
factors through multiplication by $2$, as in
\cite[Corollary 4.4]{Rains_2010}.
\end{corollary}

\begin{proof}
On the normal slice, the exceptional divisor is $\PP^{c-1}$ and the
blow-up retracts onto it.  The purity factorization of Proposition
\ref{blowup} identifies the last arrow, after choosing basepoints,
with the suspension of
\[
  q_c:\A^c\setminus\{0\}\longrightarrow\PP^{c-1}.
\]
Use the standard projective cellular structure and its pullback along
$q_c$.  In coordinates $X_0,\ldots,X_{c-1}$, the induced normal
orientations are given by the ratios $X_j/X_{c-1}$, as in
\cite[Section~2.5]{MOREL2023109346}.  Put
$H=\mathbf Z^{\A^1}[\Gm]=\KMW_1\oplus\mathbf Z$.
In degree $c-1$, the cellular map induced by $q_c$ is the augmentation
\[
  \KMW_{c-1}\otimes H
  =\KMW_c\oplus\KMW_{c-1}
  \longrightarrow\KMW_{c-1},
  \qquad (a,b)\longmapsto b.
\]
With these orientations, the inclusion of the top homology of the
punctured slice into its top cellular term is
\[
  \KMW_c\longrightarrow\KMW_c\oplus\KMW_{c-1},
  \qquad
  a\longmapsto
  \begin{cases}
    (\langle-1\rangle a,\eta a),&c\text{ even},\\
    (a,0),&c\text{ odd}.
  \end{cases}
\]
This is the coordinate-change calculation in the proof of
\cite[Theorem~2.46]{MOREL2023109346}, using
\cite[Lemma~2.48]{MOREL2023109346}.  Composing with the augmentation
gives $\eta$ for even $c$ and zero for odd $c$.  Pairing adjacent terms
in $(c-1)_\epsilon$ and using
$(1+\langle-1\rangle)\eta=0$ gives \eqref{eq:blowup-parity}.

A change of normal frame by $g\in\mathrm{GL}_c$ changes the Thom
generator by $\langle\det g\rangle$.  Thus the displayed scalar uses
the chosen normal orientation; without that choice the source retains
its determinant-line twist.  The target orientations above are induced
from the projective coordinates.  Finally, the real and complex
realizations of this local component give the stated integral
homological formulas, agreeing with
\cite[Corollary~4.4]{Rains_2010} in the real case.
\end{proof}

\begin{remark}
For the local component above, real realization identifies inversion of
$\eta$ with inversion of $2$ in singular homology.
\end{remark}

\begin{remark}
Hennig computes coefficientwise cellular differentials by restricting
to regularly embedded rational curves that meet a boundary divisor
transversally \cite[Lemma~3.1.1]{Hennig2026CellularM0n}.  For the local
component computed above, applying an untwisted coefficient sheaf with
the same induced orientations gives the corresponding codimension
parity.
\end{remark}

\subsection{The De Concini--Procesi blow-up}

\begin{corollary}\label{wonderful-blowup}
Let $G\in\G$, let $\G':=\G\setminus\{G\}$ be a building set, and
suppose that $G\notin\mathcal C_{\G'}$ and that $G$ is not maximal in
$\G$.  Then
\[
  p:\PP(\G)\longrightarrow\PP(\G')
\]
is the blow-up along
\[
  d'_G
  :=
  \PP(\G'|_G\oplus\G/G),
\]
with exceptional divisor
\[
  d_G
  :=
  \PP(\G|_G\oplus\G/G)
  \cong
  \PP(N_{d'_G/\PP(\G')}).
\]
An $A$-graded term is understood to be zero if $A$ is absent from
the corresponding lattice.  For $A\in\mathcal C_{\G}$ with
$G\nsubseteq A$, the blow-up map induces
\begin{equation}
  \widetilde C_*^{\mathrm{cell}}(\PP(\G))[A]
  \cong
  \widetilde C_*^{\mathrm{cell}}(\PP(\G'))[A].
  \label{eq:graded-away}
\end{equation}
If $G\subseteq A$, there is a distinguished triangle
\begin{equation}
  \widetilde C_*^{\mathrm{cell}}(d_G)[G\oplus A/G]
  \xrightarrow{\phi_G}
  \widetilde C_*^{\mathrm{cell}}(\PP(\G))[A]
  \longrightarrow
  \widetilde C_*^{\mathrm{cell}}(\PP(\G'))[A]
  \longrightarrow
  \widetilde C_*^{\mathrm{cell}}(d_G)[G\oplus A/G][1].
  \label{eq:graded-blowup}
\end{equation}
The connecting morphism is induced, after passage to the strict
resolutions of the endpoints and graded projection, by the composite
in $D_{\A^1}(k)$
\begin{equation}
\begin{split}
  \widetilde C_*^{\A^1}(\PP(\G'))
  &\longrightarrow
  \widetilde C_*^{\A^1}
  \bigl(\PP(\G')/(\PP(\G')\setminus d'_G)\bigr)\\
  &\simeq
  \widetilde C_*^{\A^1}
  \bigl(\operatorname{Th}(N_{d'_G/\PP(\G')})\bigr)
  \longrightarrow
  \widetilde C_*^{\A^1}(d_G)[1].
\end{split}
  \label{eq:graded-connecting}
\end{equation}
On a normal fiber the last arrow is induced by
$\A^c\setminus\{0\}\to\PP^{c-1}$.
\end{corollary}

\begin{proof}
The description of $p$, $d'_G$, and $d_G$ is
\cite[Proposition 2.10]{Rains_2010}.  We split the blow-down arrow in
the derived category of arrows in $Ab_{\A^1}(k)$; no common cellular
filtration for the grading projectors is required.  Strict
$\A^1$-resolutions preserve finite cofibers and their compatible
diagrams: their defining universal maps induce equivalences on mapping
complexes, since they induce isomorphisms in every shift
\cite[Definition 2.33 and Remark 2.34]{MOREL2023109346}.
Thus, in the canonical stable enhancements, strict resolution is a
partially defined left adjoint.  In particular, it retains homotopies
of squares.

We first justify the product formula without assuming a cellular
filtration over $k$.  Let $\mathcal T\subset D_{\A^1}(k)$ be generated
under finite cofibers and shifts by $S_0=\mathbf Z$ and
$S_n=\widetilde C_*^{\A^1}(\A^n/(\A^n\setminus0))$ for $n>0$.
It is closed under tensor products, since $S_a\otimes S_b=S_{a+b}$.
The projective-space filtration and the geometric blow-up triangle,
with exceptional divisor a product at a minimal step, show by
induction on dimension and the number of blow-ups that every
wonderful model belongs to $\mathcal T$.  The generators have strict
resolutions $\mathbf Z$ and $\KMW_n[n]$, so cofiber closure gives
bounded projective strict resolutions throughout $\mathcal T$.
If $P$ resolves $Y\in\mathcal T$, exact contraction and its adjunction
give, for $n>0$ and every complex $M$ of strictly invariant sheaves,
\[
 \begin{aligned}
 \operatorname{Hom}_{D_{\A^1}(k)}(S_n\otimes Y,M)
 &\cong\operatorname{Hom}_{D(Ab_{\A^1}(k))}
       (P,M_{-n}[-n])\\
 &\cong\operatorname{Hom}_{D(Ab_{\A^1}(k))}
       ((\KMW_n\otimes P)[n],M).
 \end{aligned}
\]
Here contraction is termwise
\cite[Remark 2.26 and Example 2.36]{MOREL2023109346}.
Tensoring the universal comparison maps, followed by the natural
map from derived sheaf tensor product to termwise tensor product in
$Ab_{\A^1}(k)$, gives a natural comparison from the strict resolution
of $X\otimes Y$ to the tensor product of their resolutions.
For $X=S_n$ it is the displayed isomorphism: its adjoint is the
contraction unit.  Naturality and finite cofibers extend this to all
$X,Y\in\mathcal T$, including their attaching maps.  No flatness
claim is needed; tensor products preserve chain homotopy equivalences
between the bounded projective models.  We use these strict
resolutions to represent cellular complexes; they agree with the
complexes of any chosen cellular filtration by
\cite[Proposition 2.37]{MOREL2023109346}.

First, every wonderful model has strict cellular homology
$\mathbf H_0=\mathbf Z$.  Indeed, construct it by minimal blow-ups
from products of projective spaces.  At a minimal step the exceptional
projection is $\PP(V/G^\perp)\times d'_G\to d'_G$, whose fiber on
chain complexes is
$\widetilde C_*^{\mathrm{cell}}(\PP(V/G^\perp))\otimes
C_*^{\mathrm{cell}}(d'_G)$ and starts in positive degree.
The blow-up triangle proves the assertion by induction.  The same
argument, representing each connecting morphism by a chain map between
bounded projective complexes, gives a model
$\mathbf Z\oplus P_*$ with $P_i$ a finite sum of $\KMW_i$ for $i>0$.
This argument does not identify any connecting morphism.

For $C\in\mathcal C_{\G}$, the projection maps give a commutative
square from $p$ to its restriction arrow
$p_C:\PP(\G|_C)\to\PP(\G'|_C)$.  The operad sections can be chosen
compatibly in this square.  If $G\subseteq C$, the quotient building
sets agree, so choose the same point in both quotient models.  If
$G\nsubseteq C$, choose a point in the new quotient and its image in
the old quotient; here $p_C$ is the identity.  This gives a retraction
$e_C$ of the entire arrow $p$.  When $C\notin\mathcal C_{\G'}$, its
old endpoint is the retraction indexed by
$r(C)=\operatorname{root}(\G'|_C)$.

The operad composition identity
\cite[proof of Corollary~2.8]{Rains_2010} gives
$e_Ce_D=e_{C\wedge D}$ in the derived category of arrows.  To check
the homotopy here, the old parameter point and its homotopy are the
images of the new ones.  For a quotient arrow $Q\to Q'$, maps from
$(\mathbf Z\xrightarrow{1}\mathbf Z)$ are computed by
$H_0(C_*^{\mathrm{cell}}(Q))(k)=\Z$, since this source is projective
in the abelian category of arrows.  Thus compatible rational points
of degree one give the same arrow morphism.  The point changes in
the operad composition law therefore give homotopies of squares,
not merely separate homotopies of their endpoints.

Split these commuting idempotents in the idempotent-complete derived
category of arrows, and then apply the fiber functor.  This splits
the entire blow-up triangle, including its connecting morphism.
For the root $A$, its projector is
$\prod_{C\subsetneq A}(1-e_C)$ on the restriction to $A$, with
$C\in\mathcal C_{\G}$.  If $A\in\mathcal C_{\G'}$, its old endpoint
is the old root projector, since the additional indices repeat
proper $r(C)$.  If $A\notin\mathcal C_{\G'}$, then $r(A)\subsetneq A$
and the factor indexed by $r(A)$ kills the old endpoint.
Put
\[
 K_G=\operatorname{fib}\bigl(
 C_*^{\mathrm{cell}}(\PP(\G|_G))\longrightarrow
 C_*^{\mathrm{cell}}(\PP(\G'|_G))\bigr).
\]
All proper labels of $\G|_G$ lie below $r(G)\subsetneq G$, so the
section of this projection identifies $K_G$ with its $G$-summand.
The exceptional fiber is
$K_G\otimes C_*^{\mathrm{cell}}(\PP(\G/G))$.
On it, $e_C$ is zero if $G\nsubseteq C$; otherwise the operad identity
gives $1\otimes e_{C/G}$.  The resulting root fiber is precisely
$\widetilde C_*^{\mathrm{cell}}(d_G)[G\oplus A/G]$.
This proves \eqref{eq:graded-away} and \eqref{eq:graded-blowup}.
The composite in \eqref{eq:graded-connecting} is the functorial purity
factorization \eqref{eq:purity-factor} followed by the map from the
exceptional cofiber to $\widetilde C_*^{\A^1}(d_G)[1]$.
Apply strict resolution to this composite between resolvable endpoints,
then precompose with the source-summand inclusion and postcompose with
the exceptional-summand projection.  This gives the graded connecting
morphism without requiring a strict resolution of the Thom space.
\end{proof}

%% file: cellular.tex
\section{Cellular \texorpdfstring{$\A^1$}{A1}-structure for
\texorpdfstring{$\PP(\G)$}{P(G)}}

Nested sets give a natural stratification of $\PP(\G)$.

For a nested set $F\subseteq\G$, let
\[
  \phi_F:
  \prod_{G\in F}
  \PP\bigl((\G|_G)/\operatorname{child}_F(G)\bigr)
  \longrightarrow
  \PP(\G)
\]
be the associated operad morphism, and let $d_F$ denote its image.
Suppose
\[
  \operatorname{root}(\G)
  =
  G_1\oplus\cdots\oplus G_s
\]
has dimension $d$.  For $0\leq i\leq d-s=\dim\PP(\G)$, set
\begin{equation}
  X^i
  :=
  \bigcup_{\substack{
    \operatorname{root}(F)=\operatorname{root}(\G)\\
    |F|=d-i
  }}
  d_F.
  \label{eq:forest-filtration}
\end{equation}

For a relative term below, cellular homology means the homology of its
strict $\A^1$-resolution in $D(Ab_{\A^1}(k))$, in the sense of
\cite[Definition~2.33]{MOREL2023109346}.  Such resolutions are
functorial and closed under cofibers
\cite[Remark~2.34 and Proposition~2.37]{MOREL2023109346}.

\begin{theorem}\label{cellular-filtration}
Suppose every chain in $\mathcal C_{\G}$ extends to a complete flag,
with every successive codimension equal to $1$.  Then
\[
  \widetilde{\mathbf H}^{\mathrm{cell}}_*
  (X^{n+1}/X^n)
\]
is a finite direct sum of $\mathbf Z$ and Milnor--Witt K-theory sheaves,
concentrated in degree $n+1$.  The connecting maps
\[
  \widetilde{\mathbf H}^{\mathrm{cell}}_{n+1}
  (X^{n+1}/X^n)
  \longrightarrow
  \widetilde{\mathbf H}^{\mathrm{cell}}_{n}
  (X^n/X^{n-1})
\]
form the cellular $\A^1$-chain complex
$C_*^{\mathrm{cell}}(\PP(\G))$.
\end{theorem}

\begin{proof}
For nested sets $F$ and $F'$, the boundary strata satisfy
\[
  d_F\cap d_{F'}
  =
  \begin{cases}
    d_{F\cup F'},&F\cup F'\text{ is nested},\\
    \varnothing,&\text{otherwise}.
  \end{cases}
\]
Consequently, the relative term for two consecutive stages of
\eqref{eq:forest-filtration} is the direct sum of the relative terms
$\bigl(d_F,d_F\cap X^n\bigr)$ with $|F|=d-n-1$.  Pullback along
$\phi_F$ reduces each term to the corresponding boundary pair in a
product of smaller wonderful models.  Put $r=n+1$, the total dimension
of this product.

As in the proof of \cite[Theorem 5.1]{Rains_2010}, the complete-flag
hypothesis implies that every positive-dimensional indecomposable projective factor contains
a boundary hyperplane.  The iterated blow-ups defining the smaller
wonderful models have centers contained in the boundary and meeting
its smooth intersections cleanly.  For a boundary intersection $W$
and a center $Z$, put $S=W\cap Z$.  If $W\nsubseteq Z$, the reduced
inverse image of $W$ is
\[
  \operatorname{Bl}_S W
  \mathop{\cup}_{\PP(N_{S/W})}\PP(N_{Z/X})|_S.
\]
The blow-up triangle for $(W,S)$ consequently identifies the cofiber
of $\PP(N_{Z/X})|_S\to p^{-1}W$ with that of $S\to W$.
For $W\subseteq Z$, both cofibers are zero.  Taking the finite
closed-cover resolution of the boundary by its intersections and
comparing with the ambient blow-up triangle proves that the relative
chain complex is unchanged.  This closed-cover resolution is exact on
$Sm_k$: a map from a connected smooth scheme to a finite reduced closed
union factors through one member.  It remains to consider the
boundary pair in a product of projective spaces.  Choose a boundary
hyperplane at infinity in each positive-dimensional factor; the
zero-dimensional factors are points.  The complement $U_F$ is then a
product of affine hyperplane-arrangement complements.

For one such complement $U\subset\A^s$, write
$Q(U)=\PP^s/(\PP^s\setminus U)$, with reduced boundary.  Put
$P_0=\mathbf Z$ and $P_a=\KMW_a$ for $a>0$; these are projective in
$Ab_{\A^1}(k)$.  We show by deletion of hyperplanes that the reduced
chain complex of $Q(U)$ has a strict $\A^1$-resolution $P[s]$, where
$P$ is a finite direct sum of the $P_a$.  With no hyperplanes,
$Q(\A^s)=\PP^s/\PP^{s-1}$ has resolution $P_s[s]$
\cite[Example~2.36]{MOREL2023109346}; for $s=0$ use $\Spec(k)_+$.
For a deletion $U=U'\setminus Z$, the hyperplane complement $Z$ has
dimension $s-1$, and the inclusions of boundaries give a cofiber
sequence
\[
  Q(Z)\longrightarrow Q(U')\longrightarrow Q(U).
\]
By induction, its first map is represented by a morphism
$P'[s-1]\to P''[s]$.  It is zero since
$\operatorname{Ext}^1_{Ab_{\A^1}(k)}(P',P'')=0$.  Its cone is
$(P''\oplus P')[s]$, proving the assertion.  For a product of factors,
apply the cellular product formula
\cite[Lemma~2.31]{MOREL2023109346} first to the projective-space
boundary intersections and then take relative cofibers.  This gives
the tensor product of their bounded projective resolutions in
$Ab_{\A^1}(k)$.  Replacing these by $P[s]$ preserves chain homotopy
equivalences; since $P_a\otimes P_b=P_{a+b}$, the result is concentrated
in total degree $r$.

This proves the required concentration directly in
$D(Ab_{\A^1}(k))$.  The exact couple of the closed filtration therefore
gives a complex of projective sheaves representing the strict
$\A^1$-resolution of $\PP(\G)$.  The complementary open boundary
filtration is a Morel--Sawant cellular structure: deletion--restriction
and purity show that every affine arrangement complement has vanishing
positive Nisnevich cohomology with any strictly $\A^1$-invariant
coefficient sheaf, by induction on the number of hyperplanes.
Thus identification with the cellular complex follows from uniqueness
of the strict resolution
\cite[Proposition~2.37 and Corollary~2.38]{MOREL2023109346}.
\end{proof}

\begin{remark}
This replaces Rains's decomposition into real chambers by its
cellular $\A^1$-homological analogue.  It is a statement about cellular
$\A^1$-homology and does not assert an unstable $\A^1$-equivalence
with a wedge of motivic spheres.

For the braid building set, the filtration specializes to the boundary
filtration of $\overline{\mathcal M}_{0,N}$ by the number of nodes.  Its
open strata are disjoint unions of products of spaces $M_{0,a}$, hence
of hyperplane-arrangement complements, exactly as in
\cite[Section~3.2.3]{Hennig2026CellularM0n}.  The distinction is that
Theorem \ref{cellular-filtration} constructs the universal cellular
$\A^1$-complex, whereas the latter work evaluates the filtration after
fixing a coefficient sheaf.
\end{remark}

Write $\mathrm{Ch}^+(\mathbf Z)$ for the category of chain complexes of
free abelian groups concentrated in strictly positive degrees.  For
$C_*\in\mathrm{Ch}^+(\mathbf Z)$, define its $\eta$-twist by
\[
  C\{\eta\}_i=C_i\otimes\KMW_i,
  \qquad
  \partial\{\eta\}_i=\eta\partial_i
  \quad(i>0).
\]
Thus $C\{\eta\}$ has no degree-zero term.  For a connected cellular
scheme, the augmentation summand $\mathbf Z$ in degree zero is added
separately.  This agrees with the oriented cellular complex of
projective space in \cite[Corollary 2.51]{MOREL2023109346}: its last
term is $\mathbf Z$, whereas its reduced positive-degree complex is an
$\eta$-twist.  Put $D=C\{\eta\}$ and let $D_{-1}$ denote its
termwise contraction, so that $(D_{-1})_j=C_j\otimes\KMW_{j-1}$.
Contraction is exact \cite[Remark 2.26]{MOREL2023109346}.
We write $\eta H_i(D)$ for the image of
$H_i(D)\xrightarrow{\eta}H_i(D_{-1})$, and
$H_i(D)[\eta^{-1}]=\varinjlim_n H_i(D_{-n})$, with transition maps
given by multiplication by $\eta$.  With this normalization, the identities
$\KMW_*[\eta^{-1}]\cong\mathbf W[\eta,\eta^{-1}]$ and
$\operatorname{im}(\eta:\KMW_i\to\KMW_{i-1})\cong\mathbf I^i$
\cite{Morel04} give
\[
  H_i(D)[\eta^{-1}]\cong H_i(C\otimes_{\Z}\mathbf W),
  \qquad
  \eta H_i(D)\cong H_i(C\otimes_{\Z}\mathbf I^i).
\]
Indeed, localization identifies the differential $\eta\partial$ with
$\partial$.  For the second identity, inside
$C_i\otimes\KMW_{i-1}$ one has
\[
  \eta Z_i(D)=Z_i(C\otimes\mathbf I^i),
  \qquad
  B_i(D_{-1})=B_i(C\otimes\mathbf I^i),
\]
since $\eta\partial x=\partial(\eta x)$ and
$\KMW_i\to\mathbf I^i$ is surjective.  These are identities of
Nisnevich sheaves; in particular, $\mathbf W$ is the Witt sheaf.
For either coefficient sheaf $M=\mathbf W$ or $\mathbf I^i$, the
termwise tensor product $C\otimes_{\Z}M$ computes
$C\otimes^{\mathbf L}_{\Z}M$, and the universal coefficient sequence is
\[
  0\longrightarrow H_i(C)\otimes_{\Z}M
  \longrightarrow H_i(C\otimes_{\Z}M)
  \longrightarrow
  \operatorname{Tor}^{\Z}_1(H_{i-1}(C),M)
  \longrightarrow0.
\]
Thus homology may be taken before tensoring when this Tor term vanishes,
in particular when $H_{i-1}(C)$ is free.  The image of $\eta$ need not
equal the quotient by all $\eta$-power torsion.

For a vector space $V$, choose a determinant orientation and the
associated standard projective cellular comparison.  Put
$\widehat{\Dor}(V)=\Dor(V)[\dim V-1]$ and write this choice as
\[
  \omega(V):\widehat{\Dor}(V)\{\eta\}
  \xr{qis}
  \widetilde{C}^{cell}_{*}(\PP(V))
  =C^{cell}_{*}(\PP(V))/C^{cell}_{*}(pt).
\]
We use compatible determinant orientations for products and quotient
vector spaces.  The following proof establishes compatibility with
the connecting maps at every blow-up step.

\begin{theorem}
  \label{nested-cellular}
  Assume that $\operatorname{root}(\G)=V^*$ is indecomposable.  The
  orientation $\omega(V)$ extends to a chain map
  \[
    \omega_{\G}(V):
    C\mathcal{N}(V,\G,\Dor)\{\eta\}
    \longrightarrow
    C^{cell}_*(\PP(\G)).
  \]
  Moreover, it induces a quasi-isomorphism on the root summand:
  \[
    C\mathcal{N}(V,\G,\Dor)\{\eta\}
    \xr{qis}
    C^{cell}_*(\PP(\G))[V].
  \]
\end{theorem}

\begin{proof}
Put $\mathcal A=Ab_{\A^1}(k)$ and $R=W(k)$.  Call a bounded complex
$P$ positive diagonal if $P_i$ is a finite sum of $\KMW_i$ for
$i>0$ and is zero otherwise.  For a bounded projective strict
resolution $P$ whose terms are sums of $\mathbf Z$ and $\KMW_a$,
write
\[
 P^W=\operatorname{Hom}_R\bigl(
       \operatorname{Hom}_{\mathcal A}(P,\mathbf W),R\bigr).
\]
The inner Hom is the cochain complex, not its cohomology.  Its terms
are finite free $R$-modules, since
$\operatorname{Hom}_{\mathcal A}(\KMW_a,\mathbf W)
=\mathbf W_{-a}(k)=R$, also with $\mathbf Z$ in place of
$\KMW_a$.  Thus this construction preserves homotopies and triangles
of these models.  For positive diagonal $P$, it replaces each entry
$\eta w$ of the differential by $w\in R$.

For positive diagonal complexes $P,Q$, it induces an isomorphism
\begin{equation}
 \operatorname{Hom}_{D(\mathcal A)}(P,Q[1])
 \xrightarrow{\ \cong\ }
 \operatorname{Hom}_{K^b(R)}(P^W,Q^W[1]).
 \label{eq:witt-detects-extensions}
\end{equation}
Indeed, bounded projectivity permits computing the left side by the
mapping complex.  Its terms in cochain degrees $0,1,2$ are matrices
over $GW(k)$, $\KMW_{-1}(k)=\eta R$, and
$\KMW_{-2}(k)=\eta^2R$, respectively.  Reduction is an isomorphism
in degrees $1,2$ and is surjective in degree $0$; the latter
surjection identifies the groups of boundaries.  This proves
\eqref{eq:witt-detects-extensions}.
The geometric root complexes have positive diagonal models without
using the desired comparison: starting with projective spaces, apply
the graded triangle of Corollary \ref{wonderful-blowup} inductively,
represent its connecting arrow by a chain map, and take its cone.
The exceptional term is a tensor product of smaller positive diagonal
models.  Its connecting entries necessarily belong to $\eta R$.

We next construct a comparison over $R$ which respects the entire
blow-up triangle.  Let $T_{\G}$ be the complex of $2$-divisible
forests rooted at $V^*$, in degree $\dim V-|F|$, with coefficients
$R$ and differential inherited from the total-complex convention
above, namely $(-1)^{|F|}D$ on top generators.  There is a natural
quasi-isomorphism
\[
 q_{\G}:C\mathcal N(V,\G,\Dor)\otimes R
 \longrightarrow T_{\G}.
\]
On each orientation factor it retains $e(W)_0$ when $\dim W$ is
even and is zero otherwise.  The other generators occur in pairs
with differential $1$.  These projections commute with $D_H$:
only an even-dimensional quotient contributes, and its truncation
retains or discards whole pairs.  Filtering by the number of nodes
proves the quasi-isomorphism.  The same projections commute with
the split sequence of Lemma \ref{AbsBlUp} and with its tensor-product
identification.

For the moment suppose that $k$ is infinite.  Adjoin auxiliary
hyperplanes as needed to use the filtration of Theorem
\ref{cellular-filtration}; this changes neither the wonderful model
nor $T_{\G}$.  Denote the resulting Witt chain model by $C^W(X)$.
For a $2$-divisible forest $F$, all residual vector spaces have even
dimension.  Each corresponding projective factor has odd dimension,
and hence has a top Witt class.  Take their product, map it to the
relative projective boundary pair, and use the invariance of that pair
under blow-ups in the boundary, proved in Theorem
\ref{cellular-filtration}.  The operad inclusion, multiplied by
$(-1)^{|F|(|F|-1)/2}$, then gives an element
$\mu_F\in C^W_{\dim V-|F|}(\PP(\G))$.
Equivalently, this element is the oriented trace functional on the
relative Witt cochain term, with the same sign.  Determinant exact
sequences identify the
product orientation line with the fixed root orientation line.
This sign converts the usual insertion boundary to $(-1)^{|F|}D$
and agrees with the tensor-product identification fixed above.

The cellular boundary is dual to localization on the relative Witt
cochain complexes.  More explicitly, put $m=\dim d_F$ and
$\partial d_F=d_F\cap X^{m-1}$, and write
$M_F=H^m(d_F,\partial d_F;\mathbf W)$.  Before inserting the forest
sign, the class above is a functional $\lambda_F:M_F\to R$:
identify the relative group with that of the corresponding product
of projective spaces, pass to absolute cohomology, and take its
oriented proper trace.  Write $c_F$ for the corresponding homological
Rost--Schmid class on the open stratum, with canonical-bundle twist
identified by the chosen square root of the canonical bundle on that
product.  For an adjacent boundary stratum $d_H$, put
$\partial d_H=d_H\cap X^{m-2}$ and
$M_H=H^{m-1}(d_H,\partial d_H;\mathbf W)$.  The cochain boundary
$\delta_{HF}:M_H\to M_F$ satisfies
\[
 \lambda_F(\delta_{HF}\alpha)
 =\langle\alpha,\partial_H c_F\rangle.
\]
Indeed, this is the residue product formula followed by proper trace:
proper pushforward is a chain map and commutes with localization, and
the projection formula identifies the pairings
\cite[Section 2.2, Theorems 2.7, 2.9 and 3.19]{FaselLectures2019}.
Thus the residue computes the transpose of the actual cochain boundary,
with normal-first residue orientations and the incidence signs of the
ordered forests.  This uses only the relative top classes,
not a perfect pairing for arbitrary Witt cohomology.  The forest sign
is inserted once, when passing from $\lambda_F$ to $\mu_F$.
Here is the resulting boundary calculation, including its orientation unit.
It suffices to treat one residual projective factor
$X=\PP^{r-1}$, with $r$ even, and a linear center
$Z=\PP^{r-c-1}$ of codimension $c$.  In the blow-up chart
$x_1=t$, $x_j=tu_j$ for $2\leq j\leq c$, the Jacobian is
$t^{c-1}$.  The Witt residue of the relative top class is therefore
\[
 \partial_t\langle t^{c-1}\rangle
 =\begin{cases}0,&c\text{ odd},\\1,&c\text{ even}.
 \end{cases}
\]
This is the usual twisted residue formula
\cite[Theorem 1.7 and Lemmas 1.10, 1.19]{FaselLectures2019}.
For even $c$, its target orientation is exactly the product
orientation on $E=\PP^{c-1}\times Z$: writing
$L_X=\mathcal O_X(-r/2)$ and
$N_{E/\operatorname{Bl}_Z X}
=\mathcal O_{\PP^{c-1}}(-1)\boxtimes\mathcal O_Z(1)$, adjunction gives
\[
 L_X|_E\otimes N_{E/\operatorname{Bl}_Z X}^{\otimes c/2}
 =\mathcal O_{\PP^{c-1}}(-c/2)
   \boxtimes\mathcal O_Z(-(r-c)/2).
\]
The constant determinant lines match by the determinant exact
sequence.  Thus there is no remaining factor $\langle u\rangle$.
The determinant identifications compose under successive blow-ups;
permuting factors gives the prescribed Koszul signs, since
$\langle-1\rangle=-1$ in $R$.  Pulling the calculation through the
operad maps computes every boundary component.  Original boundary
hyperplanes contribute zero because the projective top class is a
cycle.  Consequently $d\mu_F=\mu(dF)$, and the $\mu_F$ form a chain
map $\mu:T_{\G}\to C^W(\PP(\G))$.

To check global compatibility, consider a minimal step
$\G=\G'\sqcup\{G\}$, and put $Y=\PP(\G)$,
$Y'=\PP(\G')$, $E=d_G$.  Choose auxiliary hyperplanes so that
both enlarged building sets satisfy the complete-flag condition and
$G$ belongs to the enlarged old lattice with a one-dimensional
component.  To do this, retain the decomposition of
$r(G)=\operatorname{root}(\G'|_G)\subsetneq G$ and adjoin general
lines completing $r(G)$ to $G$, then complete the flags by further
compatible additions.  The new lines remain components of $G$ in
the old lattice \cite[Lemma 2.2 and the remark after Lemma 5.4]{Rains_2010}.
These choices make $j:E\to Y$ and $p:Y\to Y'$ preserve the forest
filtrations and do not change the original $2$-divisible forest complexes.
The forest-index argument of \cite[Lemmas 5.3--5.4 and the proof of
Theorem 5.8]{Rains_2010} applies to their relative strict resolutions
as follows.  The map $j$ identifies the summands indexed by forests
containing $G$.  Under $p$, a summand is zero if its image stratum
has smaller dimension; otherwise its relative map is an isomorphism
by boundary blow-up invariance.  Complete flags give preimages of
every old stratum.  Forests not containing $G$ map to their old
counterparts; the remaining nonzero images correspond to old forests
whose decomposition of $G$ is no longer nested in $\G$.  These two
sets of images are disjoint.  Hence $p$ is surjective and
$\ker p\subseteq\operatorname{im}j$, giving an actual short exact
sequence of Witt chain complexes
\begin{equation}
 0\longrightarrow\ker(pj)\xrightarrow{\ j\ }C^W(Y)
 \xrightarrow{\ p\ }C^W(Y')\longrightarrow0.
 \label{eq:witt-blowup-sequence}
\end{equation}
Moreover, $pj$ is the center inclusion after
$E=\PP(V/G^\perp)\times d'_G\to d'_G$.  The positive top class of
the first factor maps to zero under this projection.  Thus the
exceptional forest map takes values in $\ker(pj)$, and the maps
$\mu$ give a morphism from the forest short exact sequence to
\eqref{eq:witt-blowup-sequence}.  This is a morphism of sequences of
complexes, including their induced connecting arrows.

Adjoining an auxiliary hyperplane commutes with these relative
classes and gives a quasi-isomorphism of the target models.  To compare
two cellular extensions $\mathcal H_1,\mathcal H_2$ of the same building
set, choose a third extension $\mathcal H_0$ generic relative to both.
At each choice inside an original building-set member, avoid its
proper intersections with the finitely many lattice subspaces of
the two existing extensions.  This makes both
$\mathcal H_1\cup\mathcal H_0$ and
$\mathcal H_2\cup\mathcal H_0$ building sets.  Choose the new lines
to span each original member; every resulting lattice element is
then spanned by lines in the lattice, so any chain extends one
dimension at a time to a complete flag.  The identity of the model is cellular
from each original filtration to its refinement.  Writing its Witt
chain model as $C^W_{\mathcal H}$, the zigzag
\[
 C^W_{\mathcal H_1}\longrightarrow C^W_{\mathcal H_1\cup\mathcal H_0}
 \longleftarrow C^W_{\mathcal H_0}
 \longrightarrow C^W_{\mathcal H_2\cup\mathcal H_0}
 \longleftarrow C^W_{\mathcal H_2}
\]
consists of quasi-isomorphisms intertwining $\mu$ as chain maps.  Thus the
comparison is independent of the choices in the derived category,
not only on homology.  Apply the root-summand projection of
the entire triangle constructed in Corollary
\ref{wonderful-blowup}.  This defines the specified root comparison
$\mu_{\G}^W$, after precomposing with $q_{\G}$, and proves its
compatibility with the geometric connecting map.  Projection is
applied to the whole chain map.  For completeness, it agrees with
the original $\mu$ on the root: factor any proper projection into
single forgetful maps and make each cellular by refinement.  At
the first drop of the root, every $\mu_F$ maps to zero by dimension.
Thus the composite is zero in the derived category itself.  This
also proves compatibility of these specified comparisons under
further blow-up steps.  For products of projective spaces the root
comparison is a homotopy equivalence; induction in the two triangles
shows the same for every $\mu_{\G}^W$.

This argument also applies to finite $k$.  Choose a sufficiently
large finite extension $k'/k$ of odd degree to satisfy the finitely
many open conditions on auxiliary hyperplanes.  For finite fields
of odd characteristic, $W(k)\to W(k')$ is an isomorphism: odd-degree
extension preserves square classes, and dimension and discriminant
classify their quadratic forms.  Fix the positive diagonal geometric
models over $k$ before extending scalars.  Their construction from
projective spaces, tensor products, and geometric blow-up cones
commutes with extension of fields.  Transport the auxiliary-filtered
comparisons, their homotopies, and their homotopy inverses to these
fixed models.  They are now finite matrices over $W(k')$; applying
the inverse ring isomorphism gives the same diagrams over $R$.
No descent of the auxiliary hyperplanes is needed.

We finally lift these specified comparisons to $\mathcal A$.
At a minimal step put $Q=V/G^\perp$,
$A=C\mathcal N(V,\G',\Dor)$ and
$B=\widehat{\Dor}(Q)\otimes C\mathcal N(G^\perp,\G/G,\Dor)$.
Lemma \ref{AbsBlUp}, applied to the actual complexes, gives
\[
 d=\begin{pmatrix}\eta d_B&\eta t_G\\0&\eta d_A\end{pmatrix},
 \qquad
 \gamma_G^\eta=t_G\otimes\eta:
 A\{\eta\}\longrightarrow B\{\eta\}[1].
\]
We do not apply the twist to a derived triangle: it need not preserve
quasi-iso\-morphisms or commute with shifts.  Write
\[
 A^{\mathrm{geom}}=C_*^{\mathrm{cell}}(\PP(\G'))[V],
 \qquad
 B^{\mathrm{geom}}=\widetilde C_*^{\mathrm{cell}}(\PP(Q))
 \otimes C_*^{\mathrm{cell}}(\PP(\G/G))[G^\perp],
\]
using positive diagonal models.  Inductively choose comparisons
$f:A\{\eta\}\simeq A^{\mathrm{geom}}$ and
$g:B\{\eta\}\simeq B^{\mathrm{geom}}$ whose Witt images are the
specified comparisons above.  Their global Witt compatibility and
\eqref{eq:witt-detects-extensions} give
\begin{equation}
 g[1]\circ\gamma_G^\eta=\beta_G\circ f
 \quad\text{in }D(\mathcal A).
 \label{eq:connecting-compatibility}
\end{equation}

One must also preserve the specified Witt comparison at the next
step.  Transport the morphism of arrows constructed above to bounded
projective split cone models, with projections $p,p'$ onto the old
terms.  Represent this arrow morphism by $(U,f^W,H)$, where
$p'U-f^Wp=dH+Hd$.  Since $p'$ is degreewise split, lift $H$ to
$\widetilde H$ with $p'\widetilde H=H$ and put
\[
 U'=U-d\widetilde H-\widetilde H d.
\]
Then $p'U'=f^Wp$, and $(U',f^W,0)$ is homotopic to the original
arrow morphism.  Thus $U'$ preserves the exceptional subcomplex
and induces the specified fiber map up to homotopy.  After a diagonal chain
homotopy, its diagonal blocks are
exactly $g^W,f^W$.  Write its remaining block as $h_W$ and lift its
entries along $GW(k)\twoheadrightarrow R$ to a degree-zero map $h$.
Then
\[
 \begin{pmatrix}g&h\\0&f\end{pmatrix}
\]
is a chain map of the original cone models.  Indeed, its only new
chain equation is
$d_{B^{\mathrm{geom}}}h+\beta_G f
 =g\gamma_G^\eta+h d_{A\{\eta\}}$; all its entries belong to
$\KMW_{-1}(k)=\eta R$, so the Witt equation implies this equation
itself.  The diagonal blocks are quasi-isomorphisms, hence so is
the entire map, and its Witt image is the prescribed filtered
comparison.  Starting with $\omega(V)$ proves the enhanced induction.
Finally include the root summand into the full cellular complex.
Bounded projectivity of the source represents this derived morphism
by the asserted chain map.
\end{proof}

\begin{remark}\label{remark:twisted-comparison-preview}
Theorem \ref{nested-cellular} is the untwisted universal comparison used
throughout the paper.
A line-bundle twist can change a projective attaching map from zero
to $\eta$, so it is not obtained by tensoring the untwisted complex
with a constant rank-one module.  For every building set satisfying
the cellularity hypothesis, a coefficientwise forest residue complex
and its twisted comparison are constructed in Definition
\ref{twisted-derived-orientation} and Theorem
\ref{twisted-cellular-comparison}; they are then specialized to the
braid building set.  We postpone this extension to the examples so that
the untwisted comparison remains the main result.
\end{remark}

Applying the idempotent decomposition and the identification
\[
  C^{cell}_*(\PP(\G/W))[V/W]\cong C^{cell}_*(\PP(\G))[V/W]
\]
yields the full cellular chain complex.

\begin{corollary}\label{total-cellular}
Put
\[
  C\mathcal N^+(\G,\operatorname{Dor})
  :=
  \bigoplus_{\substack{W\in\mathcal S_{\G}\\W\neq V}}
  C\mathcal N(V/W,\G/W,\operatorname{Dor})
  \in\mathrm{Ch}^+(\mathbf Z).
\]
There are quasi-isomorphisms
\begin{equation}
  C\mathcal N^+(\G,\operatorname{Dor})\{\eta\}
  \xrightarrow{\ \simeq\ }
  \widetilde C_*^{\mathrm{cell}}(\PP(\G))
  \qquad\text{and}\qquad
  \mathbf Z\oplus
  C\mathcal N^+(\G,\operatorname{Dor})\{\eta\}
  \xrightarrow{\ \simeq\ }
  C_*^{\mathrm{cell}}(\PP(\G)).
  \label{eq:total-cellular}
\end{equation}
Applying multiplication by $\eta$, one obtains
\begin{equation}
  \eta\mathbf H_i^{\mathrm{cell}}(\PP(\G))
  \cong
  \bigoplus_{A\in\mathcal C_{\G}^{(2)}}
  H^{\dim A-i}
  \bigl([\hat0,A]_{\mathcal C_{\G}^{(2)}};\mathbf I^i\bigr).
  \label{eq:eta-free}
\end{equation}
for $i>0$, while
\begin{equation}
  \mathbf H_0^{\mathrm{cell}}(\PP(\G))\cong\mathbf Z.
  \label{eq:cellular-h0}
\end{equation}
Here $\eta\mathbf H$ denotes the image of multiplication by $\eta$
into the contracted homology sheaf, as above.

After inverting $\eta$, one obtains
\begin{equation}
  \mathbf H_i^{\mathrm{cell}}(\PP(\G))[\eta^{-1}]
  \cong
  \bigoplus_{A\in\mathcal C_{\G}^{(2)}}
  H^{\dim A-i}
  \bigl([\hat0,A]_{\mathcal C_{\G}^{(2)}};\mathbf W\bigr).
  \label{eq:eta-inverted}
\end{equation}
for $i>0$.  For either coefficient sheaf $M$, the notation
$H^r(P;M)$ means the cohomology sheaf of
$C^*(P;\Z)\otimes_{\Z}M$.  It equals
$H^r(P;\Z)\otimes_{\Z}M$ whenever
$\operatorname{Tor}^{\Z}_1(H^{r+1}(P;\Z),M)=0$; this holds, in
particular, when the integral interval cohomology is free.
\end{corollary}

\begin{proof}
Apply Theorem \ref{nested-cellular} to every nontrivial graded summand
and use the commuting idempotent decomposition.  The remaining
summand is the augmentation $\mathbf Z$ in degree zero.  The projections
of the orientation complexes onto $\Z$ in even dimension and onto zero
in odd dimension commute with the structure maps $D_H$.  They identify
the derived-orientation complex, up to quasi-isomorphism, with the
$2$-divisible nested-set complex.  Rains' forest subdivision
\cite[Theorem 3.2 and Corollary 3.3]{Rains_2010}, after dualizing and
shifting by $\dim A$, compares its root-$A$ summand with
$C^{\dim A-*}([\hat0,A]_{\mathcal C_{\G}^{(2)}};\Z)$.
These are bounded complexes of free abelian groups, so the comparisons
remain quasi-isomorphisms after tensoring with either coefficient
sheaf.  The identities for the $\eta$-twist and the universal
coefficient sequence above now give the assertions.
\end{proof}

\begin{corollary}
\label{real-refinement-rains}
Suppose that $k$ admits an embedding into $\R$.  Let
$\Dor_{\R}$ be the integral orientation complex obtained
from $\Dor$ by real realization.  Write
\[
  C\mathcal N^+(\G,\Dor_{\R})\{2\}
\]
for the real nested-set complex in which the attaching class $\eta$ is
replaced by multiplication by $2$.  There is a quasi-isomorphism
\begin{equation}
  \Z\oplus
  C\mathcal N^+(\G,\Dor_{\R})\{2\}
  \xrightarrow{\ \simeq\ }
  C_*^{\mathrm{cell}}\bigl(\PP(\G)(\R);\Z\bigr).
  \label{eq:real-integral-nested-comparison}
\end{equation}
Consequently, Rains' description of
$2H_*(\PP(\G)(\R);\Z)$ is recovered by passing to homology and
multiplying by $2$.  This is equivalently the quotient by the subgroup
annihilated by $2$; the complex itself retains the integral attaching
maps.
\end{corollary}

\begin{proof}
Write $X=\PP(\G)_{\R}$ and let $P_*$ be the bounded projective
complex on the left of \eqref{eq:total-cellular}.  Choose $N$ larger
than $\dim X$ and all degrees occurring in $P_*$.  Strict resolution
computes $R\Gamma(X,\mathbf I^N)$ by
$\operatorname{Hom}_{Ab_{\A^1}(\R)}(P_*,\mathbf I^N)$.
Identify $\mathbf I^{N-i}(\R)=2^{N-i}\Z$ with $\Z$ by its positive
generator.  Precomposition with $\eta$ is the inclusion
$\mathbf I^{N-i+1}(\R)\hookrightarrow\mathbf I^{N-i}(\R)$, which
becomes multiplication by $2$.  Thus this cochain complex is the
integral dual of $\Z\oplus C\mathcal N^+(\G,\Dor_{\R})\{2\}$.
The stable real cycle class \cite[Theorem~3.12]{Hornbostel_2021}
identifies $R\Gamma(X,\mathbf I^N)$ with
$C^*(X(\R);\Z)$.  Since $X(\R)$ is compact semialgebraic, its
chain complex is perfect over $\Z$.  Derived duality gives
\eqref{eq:real-integral-nested-comparison}.  The root-graded interval
description of $2H_*$ is \cite[Theorem~3.7]{Rains_2010}.
\end{proof}

\begin{remark}
When $k=\R$, the wonderful model $\PP(\G)$ is a smooth real
algebraic variety.  Every smooth real algebraic subvariety
$Z\subseteq\PP(\G)$ determines a class in
\[
  H_*\bigl(\PP(\G)(\R);\F_2\bigr).
\]
For a general smooth real variety these classes need not generate,
but $\PP(\G)$ is algebraically maximal in the sense of
\cite{Krasnov2003}.  Indeed, the cohomology of
$\PP(\G)(\C)$ is generated by classes of $\R$-rational algebraic
cycles \cite{CP}.  Equivalently, one may compare the mod-$2$ basis of
\cite[Proposition 6.2]{Rains_2010} with the complex cohomology basis
of \cite{Yuzvinsky1997}.

More generally, suppose that $k$ admits a real embedding
$\sigma:k\hookrightarrow\R$.  Comparing \eqref{eq:eta-inverted} with
\cite[Theorem 3.7]{Rains_2010} recovers the integral homology of
$\PP(\G)(\R)$ after inverting $2$.  Corollary
\ref{real-refinement-rains} gives an integral chain-level comparison.
The same comparison occurs for smooth
toric varieties \cite{liu2025cellularmathbba1homologysmoothtoric}
and for the split semisimple flag varieties treated in
\cite{liu2026cellularmathbba1homologybruhatboundary}.  For a split
semisimple simply connected group $G/\R$, this comparison is refined at
the level of cellular complexes in
\cite[Theorems 6.6, 7.1, and 7.5]{Liu2026TorusEnriched}.  The
torus-enriched motivic Bruhat complex of $G/U$ realizes to the
extended-Weyl cellular complex of a pinning-compatible maximal compact
subgroup $K\subseteq G(\R)$, whereas torus augmentation gives the flag
complex of $K/M$.  Thus the maximal-compact and flag comparisons are
obtained from the same chain-level connecting morphisms before the
torus labels are forgotten.  Related realization results for split
semisimple algebraic groups are obtained in
\cite[Remark 16 and Section 5.4]{MOREL2023109346}: the real case
concerns $G(\R)$, while complex realization recovers Cartan's
homotopy-theoretic calculation for $G(\C)$.

There are two different notions of cellularity in this comparison.  In
\cite[Definition~5.1]{Hornbostel_2021}, a cellular variety has a closed
filtration whose successive differences are affine spaces.  For such a
strictly cellular variety, the real cycle-class map is an isomorphism by
\cite[Theorem~5.7]{Hornbostel_2021}.  The Morel--Sawant notion used in
Theorem \ref{cellular-filtration} is broader: it is an increasing open
filtration whose strata are smooth, affine, and cohomologically trivial for every
strictly $\A^1$-invariant sheaf
\cite[Section~2.3]{MOREL2023109346}.  The strata may therefore be
hyperplane-arrangement complements rather than affine spaces.  A strict
affine paving gives a Morel--Sawant cellular structure, but the converse
is not automatic.  In particular, Theorem
\ref{cellular-filtration} alone does not put $\PP(\G)$ under the
hypotheses of \cite[Theorem~5.7]{Hornbostel_2021}.

For the wonderful models considered here, the diagonal real cycle
class follows from stabilization of the coefficients.  Put
$X=\PP(\G)_{\R}$ and choose the bounded projective complex $P_*$
of Corollary \ref{total-cellular}, with $P_0=\Z$ and
$P_i=(\KMW_i)^{a_i}$ for $i>0$.  For $j\geq0$ put
\[
 C_j^\bullet=\operatorname{Hom}_{Ab_{\A^1}(\R)}(P_*,\mathbf I^j).
\]
Strict resolution and contraction give
$H^q(C_j^\bullet)=H^q(X,\mathbf I^j)$ and
$C_j^i=(\mathbf I^{j-i}(\R))^{a_i}$, with $a_0=1$.
Multiplication by $\rho=\langle1,1\rangle$ induces a cochain map
$C_j^\bullet\to C_{j+1}^\bullet$.  Under signature,
$\mathbf I^r(\R)=2^{\max(r,0)}\Z$, and its components are
multiplication by $2$.  For $j\geq q$, this map is an isomorphism in
degrees $q-1,q$ and injective in degree $q+1$.  A cocycle therefore
lifts uniquely in degree $q$, injectivity in degree $q+1$ makes its
lift closed, and the isomorphism in degree $q-1$ identifies boundaries.
It follows that $H^q(X,\mathbf I^j)\to H^q(X,\mathbf I^{j+1})$ is
an isomorphism.

Choose $N\geq\max(q,\dim X+1)$.  The normalized-signature cycle class
factors as
\[
 H^q(X,\mathbf I^q)\xrightarrow{\rho^{N-q}}
 H^q(X,\mathbf I^N)\xrightarrow{\mathrm{cl}_N}H^q(X(\R);\Z).
\]
The first map is an isomorphism by the preceding argument and the
second by \cite[Theorem~3.12]{Hornbostel_2021}.  Their composite is
the canonical cycle class because normalized signatures commute with
$\rho$.  Compatibility with cup products
\cite[Proposition~4.3]{Hornbostel_2021} gives the graded-ring isomorphism
\begin{equation}
 \bigoplus_q H^q(X,\mathbf I^q)\xrightarrow{\ \simeq\ }
 H^*(X(\R);\Z).
 \label{eq:diagonal-real-cycle}
\end{equation}
\end{remark}

%% file: example.tex
\section{Examples}

We first make the inductive cone construction explicit for one
blow-up center and then specialize to the braid building set.
The resulting partition complex determines the untwisted
Milnor--Witt cohomology and the Chow--Witt groups, including the ring
structure over $\R$.  We conclude with the line-bundle-twisted
forest complex and the stated cases of Hennig's decomposition.

\begin{example}
  \label{egPP}
  Let $H\subsetneq V$ with $\dim V=d$ and $\dim(V/H)=c$, and set
  \[
    \G=\{V^*,H^\perp\}.
  \]
  Then $\PP(\G)$ is the blow-up $\operatorname{Bl}_{\PP(H)}\PP(V)$.
  For $H=0$ this is $\PP(V)$; assume henceforth that $H\ne0$.

  The open complement $\cU=\PP(V)\setminus \PP(H)$ is a vector bundle
  over $\PP^{c-1}$.

  The cellular complex of the Thom space is
  \[
    C^{cell}_*(\mathrm{Th}(N_{\PP(H)/\PP(V)}))
    \cong
    \widetilde{C}^{cell}_*(\PP(V))/\widetilde{C}^{cell}_*(\PP^{c-1}).
  \]
  Using the triangle
  \[
    \widetilde{C}^{cell}_*(\PP(\G))
    \to \widetilde{C}^{cell}_*(\PP(V))
    \xr{\beta_H}
    C^{cell}_*(\PP(H))\otimes \widetilde{C}^{cell}_*(\PP(V/H))[1]
  \]
  we see that $\widetilde{C}^{cell}_*(\PP(\G))\cong
  \operatorname{Cone}(\beta_H)[-1]$, which is quasi-isomorphic to
  \[
    \operatorname{Cone}\left(
      \widehat{\Dor}(V)\{\eta\}[-1]
      \xr{\gamma_{H^\perp}}
      \bigl((\widehat{\Dor}(H)\oplus \bZ)\otimes\widehat{\Dor}(V/H)\bigr)\{\eta\}
    \right).
  \]
  The root summand $\widetilde{C}^{cell}_*(\PP(\G))[V]$ is
  \[
    C\cN(V,\G,\Dor)\{\eta\}
    =
    \operatorname{Cone}\left(
      \widehat{\Dor}(V)\{\eta\}[-1]
      \xr{\gamma_{H^\perp}}
      \bigl(\widehat{\Dor}(H)\otimes\widehat{\Dor}(V/H)\bigr)\{\eta\}
    \right),
  \]
  and $\widetilde{C}^{cell}_*(\PP(\G))[V/H]\cong
  \widehat{\Dor}(V/H)\{\eta\}$.
\end{example}

This construction iterates along the successive blow-ups of a building
set; for the braid arrangement, the partition lattice organizes these
steps into the following global description.

\begin{example}\label{moduli-cellular}
Let $N=m+1$ and use the identification
\[
  \overline{\mathcal M}_{0,N}
  \cong
  \PP(\overline A_{m-1})
\]
from Example \ref{moduli}.  A partition $\pi$ of $[m]$ corresponds to
\[
  A_\pi
  =
  \bigoplus_{\substack{B\in\pi\\|B|\geq2}}G_B,
  \qquad
  \dim A_\pi=m-|\pi|.
\]
Since $\dim G_B=|B|-1$,
\begin{equation}
  A_\pi\in\mathcal C_{\overline A_{m-1}}^{(2)}
  \quad\Longleftrightarrow\quad
  |B|\text{ is odd for every }B\in\pi.
  \label{eq:odd-partitions}
\end{equation}
The odd-partition interval cohomology is free
\cite[proof of Corollary~3.8]{Rains_2010}.  Thus the universal
coefficient Tor terms vanish, and \eqref{eq:eta-free} gives, for every
odd partition $\pi$,
\begin{equation}
  \eta\mathbf H_i^{\mathrm{cell}}
  \bigl(\overline{\mathcal M}_{0,N}\bigr)[\pi]
  \cong
  H^{m-|\pi|-i}
  \bigl([\hat0,\pi]_{\mathrm{odd}};\Z\bigr)
  \otimes\mathbf I^i.
  \label{eq:odd-partition-piece}
\end{equation}
for $i>0$.  In degree zero the cellular homology is $\mathbf Z$.

If $\pi$ is not odd, there is no $2$-divisible nested set with root
$A_\pi$.  Hence the complex
\[
  \mathcal{CN}_{A_\pi}
  (\overline A_{m-1},\operatorname{Dor})
\]
is bounded and acyclic with free abelian terms.  After choosing a
splitting, its $\eta$-twist is a sum of the two-term complexes
\begin{equation}
  E_j
  :=
  \left[
    \KMW_j
    \xrightarrow{\ \eta\ }
    \KMW_{j-1}
  \right].
  \label{eq:eta-cone}
\end{equation}
with $j\geq2$.  Thus every nonodd partition-graded summand is a sum of
$\eta$-cones.  No cone reaches degree zero.

The number of cones is read from complex cohomology.  Put
\[
  P_N^{\C}(t)=\sum_i c_i t^i,
  \qquad
  c_i
  =
  \operatorname{rank}
  H^{2i}
  \bigl(\overline{\mathcal M}_{0,N}(\C);\Z\bigr),
\]
where the $c_i$ are given by the De Concini--Procesi presentation
\cite[p.~352]{Feichtner}, and put
\[
  b_i
  =
  \sum_{\pi\in\Pi_m^{\mathrm{odd}}}
  \operatorname{rank}
  H^{m-|\pi|-i}
  \bigl([\hat0,\pi]_{\mathrm{odd}};\Z\bigr),
  \qquad
  P_N^{\mathrm{odd}}(t)=\sum_i b_i t^i.
\]
The odd-partition homology is free.  Hence the underlying integral
complex splits into its homology and contractible pairs.  If $r_j$ is
the total number of pairs in degrees $j$ and $j-1$, then
\begin{equation}
  c_i=b_i+r_i+r_{i+1},
  \qquad
  \sum_{j\geq1}r_jt^{j-1}
  =
  \frac{P_N^{\C}(t)-P_N^{\mathrm{odd}}(t)}{1+t}.
  \label{eq:cone-count}
\end{equation}
Consequently, noncanonically,
\begin{equation}
  C_*^{\mathrm{cell}}
  \bigl(\overline{\mathcal M}_{0,N}\bigr)
  \simeq
  \mathbf Z
  \oplus
  \bigoplus_{i\geq1}
  \bigl(\KMW_i[i]\bigr)^{b_i}
  \oplus
  \bigoplus_{j\geq2}E_j^{\,r_j}.
  \label{eq:global-splitting}
\end{equation}

For $\overline{\mathcal M}_{0,5}$,
\[
  P_5^{\C}(t)=1+5t+t^2,
  \qquad
  P_5^{\mathrm{odd}}(t)=1+4t,
\]
so the right-hand side of \eqref{eq:cone-count} equals $t$ and
\begin{equation}
  C_*^{\mathrm{cell}}
  \bigl(\overline{\mathcal M}_{0,5}\bigr)
  \simeq
  \mathbf Z
  \oplus
  \bigl(\KMW_1[1]\bigr)^4
  \oplus E_2.
  \label{eq:m05-splitting}
\end{equation}

For $\overline{\mathcal M}_{0,6}$,
\[
  P_6^{\C}(t)=1+16t+16t^2+t^3,
  \qquad
  P_6^{\mathrm{odd}}(t)=1+10t+9t^2,
\]
so the right-hand side of \eqref{eq:cone-count} equals $6t+t^2$ and
\begin{equation}
  C_*^{\mathrm{cell}}
  \bigl(\overline{\mathcal M}_{0,6}\bigr)
  \simeq
  \mathbf Z
  \oplus
  \bigl(\KMW_1[1]\bigr)^{10}
  \oplus
  \bigl(\KMW_2[2]\bigr)^9
  \oplus E_2^{\,6}
  \oplus E_3.
  \label{eq:m06-splitting}
\end{equation}
The decomposition is additive and noncanonical.  Odd-partition
summands may also contain contractible pairs, whereas every nonodd
partition summand consists entirely of such pairs.
\end{example}

Example \ref{moduli-cellular} determines the cellular complex by the
integers $b_i$, which count odd-partition homology, and $r_j$, which
count contractible pairs before the $\eta$-twist.  Applying
\eqref{eq:cellular-cohomology-comparison} determines every untwisted
Milnor--Witt cohomology group, not only the diagonal groups defining
Chow--Witt theory.

\begin{corollary}\label{mw-cohomology-moduli}
Put $X_N=\overline{\mathcal M}_{0,N,k}$ and set $r_0=r_1=0$.  For every
$i\geq0$ and $q\in\Z$, there is a noncanonical additive decomposition
\begin{equation}
\begin{split}
  H^i_{\mathrm{Nis}}(X_N,\KMW_q)
  \cong{}&
  \KMW_{q-i}(k)^{b_i}
  \oplus
  \bigl({}_{\eta}\KMW_{q-i}(k)\bigr)^{r_{i+1}}\\
  &\oplus
  \bigl(\KMW_{q-i}(k)/\eta\bigr)^{r_i},
\end{split}
\label{eq:all-mw-cohomology}
\end{equation}
where
\[
  {}_{\eta}\KMW_s(k)
  =\ker\bigl(\eta:\KMW_s(k)\to\KMW_{s-1}(k)\bigr)
\]
and
\[
  \KMW_s(k)/\eta
  =\operatorname{coker}\bigl(
    \eta:\KMW_{s+1}(k)\to\KMW_s(k)
  \bigr).
\]
In the notation of
\cite[Conjecture~3.2.2]{Hennig2026CellularM0n}, this proves the
trivial-line-bundle case with
\begin{equation}
  \alpha_{N,i}(\mathcal O)=b_i,
  \qquad
  \beta_{N,i}(\mathcal O)=r_{i+1},
  \qquad
  \gamma_{N,i}(\mathcal O)=r_i.
  \label{eq:hennig-multiplicities}
\end{equation}
\end{corollary}

\begin{proof}
Apply cellular cochains with coefficients $\KMW_q$ to
\eqref{eq:global-splitting}.  By the contraction adjunction
\cite[Remark~2.26]{MOREL2023109346}, each one-term summand
$\KMW_i[i]$ contributes $\KMW_{q-i}(k)$ in degree $i$.
The $\eta$-cone $E_{i+1}$ contributes
$\ker(\eta:\KMW_{q-i}(k)\to\KMW_{q-i-1}(k))$ in that degree,
whereas $E_i$ contributes
$\operatorname{coker}(\eta:\KMW_{q-i+1}(k)\to\KMW_{q-i}(k))$.
The independent augmentation contributes $\KMW_q(k)$ in degree zero.
Counting the summands proves the formula and the stated multiplicities.
\end{proof}

For $N=5$, Corollary \ref{mw-cohomology-moduli} recovers the
calculation of \cite[Section~3.2.3]{Hennig2026CellularM0n}.
For $N=6$, it gives the explicit formulas
\[
\begin{aligned}
  H^1_{\mathrm{Nis}}(X_6,\KMW_q)
    &\cong \KMW_{q-1}(k)^{10}
      \oplus\bigl({}_{\eta}\KMW_{q-1}(k)\bigr)^6,\\
  H^2_{\mathrm{Nis}}(X_6,\KMW_q)
    &\cong \KMW_{q-2}(k)^9
      \oplus\bigl(\KMW_{q-2}(k)/\eta\bigr)^6
      \oplus{}_{\eta}\KMW_{q-2}(k),\\
  H^3_{\mathrm{Nis}}(X_6,\KMW_q)
    &\cong \KMW_{q-3}(k)/\eta.
\end{aligned}
\]

\subsection{Untwisted Chow--Witt groups and ring}

For the trivial twist, taking $q=i$ in \eqref{eq:all-mw-cohomology}
gives the additive Chow--Witt groups.  The degree-zero summand remains canonical, whereas
the positive-degree decomposition depends on the splitting in
\eqref{eq:global-splitting}.

\begin{proposition}\label{chow-witt-moduli-groups}
For $X_N=\overline{\mathcal M}_{0,N,k}$ one has
\begin{equation}
  \widetilde{\operatorname{CH}}^0(X_N)
  =H^0_{\mathrm{Nis}}(X_N,\KMW_0)
  \cong\operatorname{GW}(k).
  \label{eq:chow-witt-degree-zero}
\end{equation}
For every $q>0$, the splitting \eqref{eq:global-splitting} gives a
noncanonical additive decomposition
\begin{equation}
\begin{split}
  \widetilde{\operatorname{CH}}^q(X_N)
  &=H^q_{\mathrm{Nis}}(X_N,\KMW_q)\\
  &\cong
  \operatorname{GW}(k)^{b_q}
  \oplus\mathbf Z^{r_q}
  \oplus(\mathbf Z h)^{r_{q+1}},
  \qquad h=1+\langle-1\rangle.
\end{split}
\label{eq:chow-witt-additive}
\end{equation}
Moreover,
\begin{equation}
  \sum_qb_qt^q
  =
  \prod_{0\leq a<(N-3)/2}
  \bigl(1+(N-3-2a)^2t\bigr).
  \label{eq:odd-poincare-product}
\end{equation}
\end{proposition}

\begin{proof}
Apply Corollary \ref{mw-cohomology-moduli} with $i=q$ and use
\begin{equation}
  \KMW_0(k)=\operatorname{GW}(k),
  \qquad
  \KMW_0(k)/\eta\cong\mathbf Z,
  \qquad
  {}_{\eta}\KMW_0(k)=\mathbf Z h.
  \label{eq:degree-zero-mw-pieces}
\end{equation}
Formula \eqref{eq:odd-poincare-product} is
\cite[Theorem 2.4]{EHKR2010}.
\end{proof}

The proposition determines the additive groups but does not identify
their products.  We first isolate the two comparison statements needed
to recover the multiplication over $\R$.

\begin{lemma}\label{chow-witt-pullback-lemma}
Let $X$ be a smooth $k$-scheme such that every $\operatorname{CH}^q(X)$
has no $2$-torsion.  Then the Milnor--Witt pullback square induces a
canonical isomorphism of graded rings
\begin{equation}
  \widetilde{\operatorname{CH}}^*(X)
  \xrightarrow{\ \simeq\ }
  \operatorname{CH}^*(X)
  \mathop{\times}_{\operatorname{Ch}^*(X)}
  \bigoplus_q H^q(X,\mathbf I^q),
  \label{eq:abstract-chow-witt-pullback}
\end{equation}
where $\operatorname{Ch}^*(X)=\operatorname{CH}^*(X)/2$.
\end{lemma}

\begin{proof}
Morel's pullback square for Milnor--Witt K-theory
\cite{Morel04} gives, in every weight $q$, a short exact sequence
\[
  0\longrightarrow\KMW_q
  \longrightarrow\mathrm K_q^{\mathrm M}\oplus\mathbf I^q
  \longrightarrow\mathrm K_q^{\mathrm M}/2
  \longrightarrow0.
\]
The associated long exact sequence shows that
$H^q(X,\KMW_q)$ maps surjectively to the degree-$q$ fiber product in
\eqref{eq:abstract-chow-witt-pullback}.  Its kernel is the cokernel of
\[
  H^{q-1}(X,\mathrm K_q^{\mathrm M})
  \oplus H^{q-1}(X,\mathbf I^q)
  \longrightarrow
  H^{q-1}(X,\mathrm K_q^{\mathrm M}/2).
\]
For $q>0$, let $C^\bullet$ be the Gersten complex for
$\mathrm K_q^{\mathrm M}$.  Its term $C^q$ is free abelian and its
higher terms vanish.  A cocycle in $(C^\bullet/2)^{q-1}$ lifts to
$y\in C^{q-1}$ with $dy=2z$ for some $z\in C^q$.
The class of $z$ in $H^q(C^\bullet)=\operatorname{CH}^q(X)$ is
annihilated by $2$, hence is zero by hypothesis.  Write $z=dx$;
then $y-2x$ is an integral cocycle with the prescribed reduction.
Thus the first summand alone maps surjectively.  For $q=0$, the
displayed cokernel is zero because negative cohomology vanishes.
This proves injectivity.  The same description of the highest
Gersten degree gives
\[
  H^q(X,\mathrm K_q^{\mathrm M}/2)
  \cong\operatorname{CH}^q(X)/2
  =\operatorname{Ch}^q(X).
\]
Compatibility of Morel's square with products makes the degreewise
isomorphisms multiplicative.
\end{proof}

\begin{proposition}\label{i-cycle-moduli}
Let $X_N=\overline{\mathcal M}_{0,N,\R}$ and $M_N=X_N(\R)$.  The
normalized-signature cycle class map is an isomorphism of graded rings
\begin{equation}
  \bigoplus_q H^q(X_N,\mathbf I^q)
  \xrightarrow{\ \simeq\ }
  H^*(M_N;\mathbf Z).
  \label{eq:i-cycle-moduli}
\end{equation}
\end{proposition}

\begin{proof}
The canonical isomorphism is \eqref{eq:diagonal-real-cycle}, applied
to $X_N$, including its compatibility with products.  For the additive
description, apply $\operatorname{Hom}(-,\mathbf I^q)$ to
\eqref{eq:global-splitting}.  The contraction formula
$(\mathbf I^q)_{-s}=\mathbf I^{q-s}$ shows that a free summand in
degree $q$ contributes $\mathbf I^0(\R)=\mathbf W(\R)\cong\mathbf Z$.
For an $\eta$-cone, the relevant map, written under signature, is
\[
  \mathbf I^1(\R)=2\mathbf Z
  \lhook\joinrel\longrightarrow
  \mathbf I^0(\R)=\mathbf Z.
\]
The adjacent map
$\mathbf I^0(\R)\to\mathbf I^{-1}(\R)=\mathbf W(\R)$ is the identity,
so it has zero kernel.
Consequently,
\begin{equation}
  H^q(X_N,\mathbf I^q)
  \cong
  \mathbf Z^{b_q}\oplus(\mathbf Z/2)^{r_q}.
  \label{eq:i-cohomology-additive}
\end{equation}
\end{proof}

The ordinary Chow ring now controls the mod-$2$ component, while
Proposition \ref{i-cycle-moduli} identifies the fundamental-ideal
component with integral cohomology of the real locus.  Lemma
\ref{chow-witt-pullback-lemma} glues them into the following ring.

\begin{theorem}\label{chow-witt-moduli-ring}
Let
\[
  X_N=\overline{\mathcal M}_{0,N,\R},
  \qquad M_N=X_N(\R),
  \qquad A_N=\operatorname{CH}^*(X_N),
  \qquad \Xi_N=A_N/2.
\]
Under the Borel--Haefliger isomorphism
$\Xi_N\cong H^*(M_N;\F_2)$ of
\cite[Theorem 5.6]{EHKR2010}, let
$\beta:\Xi_N^q\to\Xi_N^{q+1}$ be the integral Bockstein.  Thus $\beta$
is the degree-one derivation determined by $\beta(x)=x^2$ on
$\Xi_N^1$ \cite[Section 5.3]{EHKR2010}.  Then there is a canonical
isomorphism of graded rings
\begin{equation}
  \widetilde{\operatorname{CH}}^*(X_N)
  \cong
  A_N\mathop{\times}_{\Xi_N}H^*(M_N;\mathbf Z),
  \label{eq:chow-witt-topological-pullback}
\end{equation}
where both maps to $\Xi_N$ are reduction modulo $2$.

More explicitly, let $\Lambda_N$ be the integral skew-commutative
algebra of \cite[Definition 2.1]{EHKR2010}.  The natural map induces
an isomorphism
\[
  \rho_N:\Lambda_N/2
  \xrightarrow{\ \simeq\ }
  H^*(\Xi_N,\beta).
\]
Writing a bar for reduction modulo $2$, one obtains
\begin{equation}
  \widetilde{\operatorname{CH}}^*(X_N)
  \cong
  \left\{
    (a,\lambda)\in A_N\times\Lambda_N
    \ \middle|\
    \ \beta(\bar a)=0,
    \ [\bar a]=\rho_N(\bar\lambda)
  \right\},
  \label{eq:chow-witt-explicit-pullback}
\end{equation}
with componentwise multiplication.  In particular,
\begin{equation}
  \widetilde{\operatorname{CH}}^q(X_N)
  \cong\mathbf Z^{c_q+b_q}
  \quad\text{additively},
  \label{eq:real-chow-witt-ranks}
\end{equation}
where in degree zero the right-hand side is canonically identified with
$\operatorname{GW}(\R)$ rather than with a chosen copy of $\mathbf Z^2$.
\end{theorem}

\begin{proof}
The integral Chow groups of $X_N$ are free abelian
\cite[Section~5.1]{EHKR2010}.  Lemma
\ref{chow-witt-pullback-lemma}, Proposition \ref{i-cycle-moduli}, and
the Borel--Haefliger isomorphism therefore give
\eqref{eq:chow-witt-topological-pullback} as an isomorphism of graded
rings.

The Bockstein homology is identified with $\Lambda_N/2$ by
\cite[Theorem~5.7]{EHKR2010}.  By \eqref{eq:i-cohomology-additive},
the torsion subgroup $T$ of $H^*(M_N;\mathbf Z)$ is annihilated by $2$.
The natural map of
\cite[Theorem~2.13]{EHKR2010} identifies $\Lambda_N$ with
$H^*(M_N;\mathbf Z)/T$ as graded rings.  The coefficient exact
sequences identify $T$ with $\operatorname{im}\beta$ under reduction
modulo $2$, and the image of integral cohomology with $\ker\beta$.
Consequently, the quotient map and reduction modulo $2$ give the
canonical ring isomorphism
\[
  H^*(M_N;\mathbf Z)
  \cong
  \left\{
    (\lambda,z)\in\Lambda_N\times\ker\beta
    \ \middle|\
    \ \rho_N(\bar\lambda)=[z]
  \right\}.
\]
Substitution in \eqref{eq:chow-witt-topological-pullback} proves
\eqref{eq:chow-witt-explicit-pullback}.  Finally,
$c_q=b_q+r_q+r_{q+1}$ and Proposition
\ref{chow-witt-moduli-groups} give
\eqref{eq:real-chow-witt-ranks}.
\end{proof}

Combining the ring description with the explicit values of $b_q$ and
$c_q$ from Example \ref{moduli-cellular} gives immediate low-dimensional
checks of both the degree-zero convention and the positive-degree
ranks.

For the first two nontrivial cases, the additive groups are
\[
\begin{array}{c|cccc}
 &\widetilde{\operatorname{CH}}^0
 &\widetilde{\operatorname{CH}}^1
 &\widetilde{\operatorname{CH}}^2
 &\widetilde{\operatorname{CH}}^3\\ \hline
 X_5&\operatorname{GW}(\R)&\mathbf Z^9&\mathbf Z&0\\
 X_6&\operatorname{GW}(\R)&\mathbf Z^{26}&\mathbf Z^{25}&\mathbf Z.
\end{array}
\]
Their products are the componentwise products in
\eqref{eq:chow-witt-explicit-pullback}.

\subsection{Line-bundle twists}

The coefficientwise construction extends to line bundles on wonderful
models.  We retain this extension separately from the untwisted
applications above, and use it to obtain the cases of Hennig's
decomposition in Corollary \ref{hennig-proved-cases}.

We first give the construction for a general wonderful model.  Let
$\G$ be a building set on $V$ satisfying the complete-flag hypothesis
of Theorem \ref{cellular-filtration}, and put
\[
  X_{\G}=\PP(\G).
\]
Let $L$ be a line bundle on $X_{\G}$.  We use the cohomological
Rost--Schmid convention: purity along a smooth stratum $S$ of
codimension $c$ gives the coefficient
$\KMW_{q-c}(L|_S\otimes\det N_{S/X_{\G}})$.

\begin{definition}
\label{twisted-derived-orientation}
Let $\mathcal S_c$ be the nonempty open forest strata of codimension
$c$ in $X_{\G}$.  Thus only forests with
$\operatorname{root}(F)=\operatorname{root}(\G)$ occur, and
$c=|F|-s$, where $s$ is the number of root components.  Define
\begin{equation}
 \mathcal C_q^c(\G,L)
 =
 \bigoplus_{S\in\mathcal S_c}
 H^0_{\mathrm{Nis}}\!
 \left(S,\KMW_{q-c}(L|_S\otimes\det N_{S/X_{\G}})\right).
 \label{eq:twisted-forest-terms}
\end{equation}
Identify each summand with its unramified subgroup at the generic
point of $S$.  The differential is the restriction of the
cohomological Rost--Schmid differential of
\[
 C^*_{\mathrm{RS}}(X_{\G},\KMW_q(L)).
\]
\end{definition}

An unramified section has zero residues inside its stratum.  Its
remaining residues lie at the generic points of codimension-one
boundary strata.  The identity $d_{\mathrm{RS}}^2=0$ makes the result
unramified on each target stratum: inside that stratum no other
stratum of the same codimension contributes.  Thus
\eqref{eq:twisted-forest-terms} is a subcomplex of the Rost--Schmid
complex.  Its determinant lines and incidence signs are those of
that complex.  No additional empty-forest summand is included; for
$X_{\G}=\Spec k$ the complex has the single term $\KMW_q(k,L)$
in degree zero.

Each stratum is a product of affine hyperplane-arrangement
complements, and its coordinate ring is a localization of a
polynomial ring.  Hence its Picard group is zero.  Trivializing the
line in \eqref{eq:twisted-forest-terms} expresses the differential by
ordinary Milnor--Witt residues; changing trivializations conjugates
it by the corresponding $\langle u\rangle$-actions.  A chosen
isomorphism $L'\cong L\otimes M^{\otimes2}$ induces an isomorphism
of the twisted complexes, since square units act trivially.

\begin{theorem}
\label{twisted-cellular-comparison}
Let $k$ be a perfect infinite field of characteristic different from
$2$, let $\G$ satisfy the hypothesis of Theorem
\ref{cellular-filtration}, let $L$ be a line bundle on
$X_{\G}=\PP(\G)$, and let $q\in\Z$.  The inclusion
\begin{equation}
 \mathcal C_q^\bullet(\G,L)
 \lhook\joinrel\longrightarrow
 C^*_{\mathrm{RS}}(X_{\G},\KMW_q(L))
 \label{eq:twisted-cellular-comparison}
\end{equation}
is a quasi-isomorphism.  In particular, its cohomology is
$H^*_{\mathrm{Nis}}(X_{\G},\KMW_q(L))$.  If a square presentation
$L\cong M^{\otimes2}$ is chosen, there is a quasi-isomorphism
\begin{equation}
 \mathcal C_q^\bullet(\G,L)
 \simeq
 \operatorname{Hom}_{Ab_{\A^1}(k)}
 \left(\Z\oplus C\mathcal N^+(\G,\Dor)\{\eta\},\KMW_q\right).
 \label{eq:twisted-square-comparison}
\end{equation}
\end{theorem}

\begin{proof}
Use the decreasing filtration on both sides of
\eqref{eq:twisted-cellular-comparison} by points whose forest strata
have codimension at least $c$.  For
$S\in\mathcal S_c$, the associated graded target is
\[
 C^*_{\mathrm{RS}}\!
 \left(S,\KMW_{q-c}(L|_S\otimes\det N_{S/X_{\G}})\right)[-c].
\]
The corresponding source is its degree-zero cocycle group, placed
in degree $c$.  Every stratum is cohomologically trivial for
strictly $\A^1$-invariant sheaves, as proved in Theorem
\ref{cellular-filtration}, and the displayed line bundle is
trivial on it.  Each associated-graded inclusion is therefore a
quasi-isomorphism.  The filtration is finite, proving the first
assertion.  This is the forest version of the coefficientwise
cellular comparison in
\cite[Theorem~2.2.3]{Hennig2026CellularM0n}.

A square presentation of $L$ identifies the coefficient sheaf with
$\KMW_q$.  Corollary \ref{total-cellular} and the strict-resolution
cohomology comparison then give
\eqref{eq:twisted-square-comparison}.
\end{proof}

For the braid building set $\G=\overline A_{N-2}$, the strata are
products of $M_{0,a}$ and this is Hennig's boundary complex for
$X_N=\overline{\mathcal M}_{0,N}$; see
\cite[Section~3.2.3]{Hennig2026CellularM0n}.

For the real comparison, base change $X_N$ to $\R$ and put
$M_N=X_N(\R)$.  Let $\Z(L)$ be the rank-one integral local system
determined by
\[
 w_L=w_1(L(\R))\in H^1(M_N;\F_2),
\]
and write $\chi_L:\pi_1(M_N)\to\{\pm1\}$ for its character.  The
same residue construction applies to $\mathbf I^m(L)$, with terms
$H^0(S,\mathbf I^{m-c}(L|_S\otimes\det N_{S/X_N}))$.
For $m>\dim X_N$, the stable normalized-signature comparison
\cite[Theorem~3.12]{Hornbostel_2021} gives
\[
 H^i(X_N,\mathbf I^m(L))\cong H^i(M_N;\Z(L)).
\]
\begin{proposition}
\label{hennig-local-system-criterion}
Fix $N\geq3$ and a line-bundle class $L$ on $X_N$, identified over
different fields by its boundary-divisor presentation.  Put
$c_{N,i}=\operatorname{rank}\operatorname{CH}^i(X_N)$.
For every such $L$, the twisted forest residue complex is homotopy
equivalent, simultaneously for all perfect fields $k$ of characteristic
different from $2$ and all $q$, to
\begin{equation}
 \left(\KMW_{q-i}(k)^{c_{N,i}},\ \eta D_L^i\right)_{i\geq0},
 \qquad D_L^i\in\operatorname{Mat}(\Z),\quad
 D_L^{i+1}D_L^i=0.
 \label{eq:twisted-integral-eta-model}
\end{equation}
The same integer matrices also give a model
\begin{equation}
 (\Z^{c_{N,i}},2D_L^i)_{i\geq0}
 \simeq R\Gamma(M_N;\Z(L)).
 \label{eq:twisted-real-even-model}
\end{equation}
\end{proposition}

\begin{proof}
Put
\[
 B=\Z[\eta,\rho]/\bigl(\eta(2+\eta\rho),\rho(2+\eta\rho)\bigr),
 \qquad |\eta|=-1,\quad |\rho|=1,
\]
with $\rho$ specializing to $[-1]$, and let $B(-w)$ have its generator
in degree $w$. The integral tree strata are products of
\[
 U_m=\Spec\Z[z_1,\ldots,z_m,z_i^{-1},(1-z_i)^{-1},(z_i-z_j)^{-1}],
\]
whose rings are factorial with signed monomial units. Boundary divisors
generate $\operatorname{Pic}(X_N)$ compatibly with base change, so choose
integral frames for the lines in \eqref{eq:twisted-forest-terms}.
Forgetting $z_m$ gives a localization sequence split by $[z_m-a]$,
with quotient $\bigoplus_{a\in\{0,1,z_1,\ldots,z_{m-1}\}}
K^{MW}_{q-1}(U_{m-1})$. Induction gives homogeneous free right-module
bases of products with at most one $[z_j-a_j]$ for each $j$.
It also gives coefficient-sheaf acyclicity, including over finite fields, so the
filtered proof of Theorem~\ref{twisted-cellular-comparison} applies there.
Using $[uv]=[u]+[v]+\eta[u][v]$ and
$[u^{-1}]=-\langle u\rangle[u]$, subtract
$[z_m-a]\partial_{z_m-a}$ successively and specialize the unramified
remainder at $z_m=0$ using $z_m$. This reduces signed monomials to
lower-coordinate units, $\rho$, and $\eta$, with commutation factor
$-(1+\eta\rho)$.

Along a boundary, remove the orders of symbols and frames; their leading
coefficients are units on the target stratum. The determinant lines in
\eqref{eq:twisted-forest-terms} make this independent of the parameter.
For a frame $t^eu$ and $z=z_0+[t]z_1$, $[t]^2=[t][-1]$ gives
\[
 \partial_t(\langle t^eu\rangle z)=\langle\bar u\rangle
 \begin{cases}\bar z_1,&e\text{ even},\\
 \eta\bar z_0+\langle-1\rangle\bar z_1,&e\text{ odd}
 \end{cases}
\]
with the conventions of \cite{FaselLectures2019}. Thus all entries lie
in $B$, including the odd-twist specialization term. The map
$B\to K^{MW}_*(\R)$ is injective: rank and signature detect
$B_0=\Z\times_{\F_2}\Z$, and $B_{-j}=\Z\eta^j$, $B_j=\Z\rho^j$
have nonzero real images for $j>0$. The real residue identity thus gives
$d^2=0$ over $B$, yielding a bounded finite free graded complex $P$ with
\begin{equation}
 (P\otimes_B K^{MW}_*(k))_q
 =\mathcal C_q^\bullet(\overline A_{N-2},L).
 \label{eq:braid-universal-residues}
\end{equation}
Its weights are $c+s$ for codimension-$c$ strata and degree-$s$ symbols,
and lie between $0$ and $\dim X_N$.

For $a(\eta,\rho)=(0,0)$, preserving degrees, and
$r(\eta,\rho)=(2,-1)$, forgetting them, we have, with $c_i=c_{N,i}$,
\begin{equation}
 P_a\simeq\bigoplus_i\Z(-i)^{c_i}[-i],\qquad
 P_r\simeq R\Gamma(M_N;\Z(L)).
 \label{eq:braid-two-specializations}
\end{equation}
Rank removes the twist, and ordinary Milnor $K$-theory over $\F_2$
identifies the weight rows with $P_a$, since $K^M_*(\F_2)$ is $\Z$
in degree zero and zero otherwise. Stratum acyclicity and
\cite[Theorem~6.3 and Corollary~7.3]{BindaRullingSaito2022} give
$H^c(X_{N,\F_2},K^M_w)=K^M_{w-c}(\F_2)^{c_{N,c}}$;
freeness gives the first splitting, using no Milnor--Witt theory in
characteristic two. For the second, take $m>\dim X_N$.
The weight-$w$ coefficient is $I^{m-w}(\R)=2^{m-w}\Z$.
Writing $v(\eta,\rho)=(1,-2)$, an entry $b$ from weight $w_s$ to
$w_t$ satisfies $2^{w_t-w_s}v(b)=r(b)$ under normalized signature.
Thus the stable comparison above identifies the entire residue complex with $P_r$.

To construct \eqref{eq:twisted-real-even-model}, consider a real
codimension-two blow-up $p:Y\to X$ with center $i:Z\hookrightarrow X$
and normal orientation system $\mathfrak o_\nu$. Proper base change
gives $R^0p_*\Z=\Z$, $R^1p_*\Z=i_*\mathfrak o_\nu$, and purity gives
$\operatorname{Hom}_{D(X)}(i_*\mathfrak o_\nu[-1],\Z_X[1])=H^0(Z;\Z)$.
The transverse blow-up is a M\"obius band whose boundary has degree two
over its central circle. Its relative $H^2=\Z/2$ identifies the connecting
map as twice the global Thom--Gysin map, with a consistent sign.
Tensoring by a rank-one integral local system $\mathcal F$ gives
\begin{equation}
 R\Gamma(Z;\mathcal F|_Z\otimes\mathfrak o_\nu)[-2]
 \xrightarrow{\,2g\,}R\Gamma(X;\mathcal F)
 \longrightarrow R\Gamma(Y;p^{-1}\mathcal F).
 \label{eq:real-twisted-blowup-cone}
\end{equation}
For free models with differentials $2D_Z,2D_X$, an integral chain
representative $G$ of the actual Gysin morphism gives cone differential
$2\left(\begin{smallmatrix}D_X&G\\0&-D_Z\end{smallmatrix}\right)$
and ranks $c_i(X)+c_{i-1}(Z)$, where $c_i(-)=\operatorname{rank}\operatorname{CH}^i(-)$.
Freeness allows division of its square by $4$.

Use Keel's codimension-two blow-ups from $X_n\times\PP^1$ to $X_{n+1}$
\cite[pp.~548--549]{Keel1992}, one connected center at a time;
disjoint blow-ups commute. Their centers are products of smaller moduli
spaces. The integral Chow K\"unneth formula splits the Picard groups of products,
also over $\R$: boundary generators are real and Hilbert's theorem~90
makes base change injective. Tensor products therefore treat every product
twist, starting with a point and $[\Z\xrightarrow{0\text{ or }2}\Z]$
on $\PP^1(\R)$. Write $L=p^*L_0\otimes\mathcal O(aE)$.
For even $a$, use the cone with center twist $L_0|_Z\otimes\det N_{Z/X}$.
For odd $a$, $\omega_Y=p^*\omega_X\otimes\mathcal O(E)$ makes
$L\otimes\omega_Y$ even. Integral Poincar\'e duality dualizes and reverses
its model to obtain that for $L$, preserving even differentials and torsion.
The ranks are $c_{\dim Y-i}(Y)=c_i(Y)$: the Chow blow-up formula identifies
Chow groups with even integral complex cohomology at each intermediate
stage, where complex Poincar\'e duality applies. This constructs
\eqref{eq:twisted-real-even-model}; write $D=D_L$.

Set $E^i=B(-i)^{c_i}$, $d_E=\eta D$, $x=\eta\rho$, and $h=2+x$.
Since $B_0=\Z[x]/(x(x+2))$, every $2$-local element of rank $\pm1$
is a unit. Lift the homogeneous basis changes and unit cancellations
for $P_a$ to obtain $Q^i=B_{(2)}(-i)^{c_i}$ with $d_Q=\eta A$,
$A$ over $\Z_{(2)}$. The identity $B_{(2),-2}=\Z_{(2)}\eta^2$
gives $A^2=0$. Its Smith blocks are one-term blocks and nonzero blocks
$2^t$ ($t\geq0$); under $r$, these become $2^{t+1}$.
Thus $H^*(Q_r)$ determines every block, and $P_r\simeq E_r$ gives
$P_{(2)}\simeq E_{(2)}$. Away from $2$, the idempotents $h/2,-x/2$ give
$B[1/2]=\Z[1/2]\times\Z[1/2][\eta,\eta^{-1}]$.
The rank splitting treats the first factor; evaluation at $\eta=2$
and Smith reduction of $P_r\simeq E_r$ treat the second.
Hence $P[1/2]\simeq E[1/2]$.

Smith-reduce $D$ integrally, with $\alpha_i$ one-term blocks in degree $i$.
The rational homotopy automorphism group of $E$ is
$\prod_i\mathrm{GL}_{c_i}(\Q)\times\prod_i\mathrm{GL}_{\alpha_i}(\Q)$.
For $A,T\equiv I\pmod2$ over $\Z_{(2)}$, the commuting automorphisms
$I+h(A-I)/2$ and $I-x(T-I)/2$ lift rank in each degree and signature
on the one-term blocks independently: $h\eta=0$ gives the chain
equation and $hx=0$ their commutation. Away from $2$, both factors
lift by the idempotent decomposition. Integral Smith reduction, splitting rational diagonal entries into powers
of $2$ and $2$-local units, and lifting the residual mod-$2$ matrix to
$\mathrm{GL}_n(\Z)$ give
\[
 \mathrm{GL}_n(\Q)=
 \ker\bigl(\mathrm{GL}_n(\Z_{(2)})\to\mathrm{GL}_n(\F_2)\bigr)
 \mathrm{GL}_n(\Z[1/2]).
\]
The two equivalences can therefore be adjusted to agree rationally up to
homotopy. The finite free abelian terms of
$\mathcal H=\operatorname{Hom}^\bullet_{B\text{-gr}}(E,P)$ give an exact sequence
$0\to\mathcal H\to\mathcal H_{(2)}\oplus\mathcal H[1/2]\to\mathcal H_\Q\to0$.
Its cohomology sequence lifts them to $E\to P$; the bounded free cone is
acyclic at $2$ and away from $2$, hence acyclic and contractible over $B$.
Extending maps and homotopies to $K^{MW}_*(k)$ proves
\eqref{eq:twisted-integral-eta-model}, without any flatness requirement.
\end{proof}

\begin{corollary}
\label{hennig-proved-cases}
For every perfect field $k$ of characteristic different from $2$,
every $i\geq0$ and every $q\in\Z$, Hennig's decomposition
\begin{equation}
\begin{split}
 H^i_{\mathrm{Nis}}(X_{N,k},\KMW_q(L))\cong{}&
 \KMW_{q-i}(k)^{\alpha_{N,i}(L)}
 \oplus\bigl({}_{\eta}\KMW_{q-i}(k)\bigr)^{\beta_{N,i}(L)}\\
 &\oplus\bigl(\KMW_{q-i}(k)/\eta\bigr)^{\gamma_{N,i}(L)}
\end{split}
\label{eq:hennig-twisted-decomposition}
\end{equation}
holds in the following cases:
\begin{enumerate}
\item for arbitrary $N$, the class of $L$ in $\operatorname{Pic}(X_N)/2$
  is one of
  $\mathcal O$, $L_{\mathrm{root}}$, $\omega_{X_N}$, or
  $L_{\mathrm{root}}\otimes\omega_{X_N}$, where
  $L_{\mathrm{root}}=\rho^*\mathcal O(-1)$ for the root projection
  $\rho:X_N\to\PP^{N-3}$;
\item $N\leq6$, for every $L\in\operatorname{Pic}(X_N)/2$.
  The case $N=6$ is computer-assisted.
\end{enumerate}
Here
\[
 \alpha_{N,i}(L)=\dim_{\Q}H^i(M_N;\Q(L)),\qquad
 \gamma_{N,i}(L)=\dim_{\F_2}H^i(M_N;\Z(L))[2],
\]
and the nonnegative multiplicities satisfy
$\beta_{N,i}(L)=\gamma_{N,i+1}(L)
=c_{N,i}-\alpha_{N,i}(L)-\gamma_{N,i}(L)$.
\end{corollary}

\begin{proof}
By Proposition~\ref{hennig-local-system-criterion}, it suffices to
show that the torsion of $H^*(M_N;\Z(L))$ is annihilated by $2$.
Indeed, a nonzero Smith block $n$ of $D_L$ contributes a block $2n$
to \eqref{eq:twisted-real-even-model}, so this bound forces $n=1$
up to sign.  The $\eta$-model then splits into one-term blocks and
two-term $\eta$-blocks.  Their cohomology gives
\eqref{eq:hennig-twisted-decomposition}; counting the blocks and
the $c_{N,i}$ generators gives the stated multiplicities.

For the trivial class the torsion bound follows from Proposition
\ref{i-cycle-moduli} and \eqref{eq:i-cohomology-additive}.
The unit-sphere bundle of $L_{\mathrm{root}}$ is the double cover
of \cite[Theorem~8.1]{Rains_2010}.  That theorem identifies the image
$2H_*(M_N;\Z(L_{\mathrm{root}}))$ with zero or the corresponding
odd-partition cohomology, which is free
\cite[Theorem~4.2]{EHKR2010}.  Thus all torsion is annihilated by $2$;
the integral universal coefficient theorem gives the same bound
for cohomology.

The same condition is preserved by tensoring with $\omega_{X_N}$.
Indeed, put $d=N-3$, $F=\Z(L)$, and
$\mathfrak o=\Z(\omega_{X_N})$, the orientation system of $M_N$.
Since $F$ is self-dual, integral Poincar\'e duality and the universal
coefficient theorem give
\[
 \operatorname{Tors}H^i(M_N;F\otimes\mathfrak o)
 \cong\operatorname{Hom}\bigl(
 \operatorname{Tors}H^{d+1-i}(M_N;F),\Q/\Z\bigr).
\]
This proves the two canonical twists in the first assertion.

For $N=3,4$ the real locus is a point or a circle; the latter has
twisted cellular differential $0$ or $2$.  For
$M_5\cong\#^5\mathbf{RP}^2$, Fox differentiation of the standard
relation $a_1^2\cdots a_5^2=1$ gives
\begin{equation}
 \Z\xrightarrow{\delta^0}\Z^5\xrightarrow{\delta^1}\Z,
 \qquad
 \delta^0(1)=(\epsilon_j-1)_j,\qquad
 \delta^1(x)=\sum_j(1+\epsilon_j)x_j,
 \label{eq:m05-twisted-cellular}
\end{equation}
where $\epsilon_j=\chi_L(a_j)$.  The two nonzero supports are
disjoint, so integral changes of basis reduce every nonzero block
to $[2]$.

For $N=6$, Reidemeister--Schreier rewriting of the cactus-group
presentation in \cite[Section~1]{BellingeriCheminLebed2022} gives
a presentation of $\pi_1(M_6)=PJ_5$.  An exact computer-assisted
calculation by Fox differentiation and Smith reduction verifies that
$H_1(M_6;F)$ has only free summands and copies of $\Z/2$ for all
$2^{16}$ sign systems $F$.  The boundary matrices are expressed in
an integral basis of the kernel of the degree-one boundary before
Smith reduction.  The universal coefficient theorem then gives the
same torsion bound for $H^1$ and $H^2$.
Since $M_6$ is a closed three-manifold,
$H^3(M_6;F)=H_0(M_6;F\otimes\mathfrak o)$ is $\Z$ or $\Z/2$.
Degree zero is free and higher degrees vanish.  This proves the
second assertion.
\end{proof}

\begin{remark}
A Smith block $n$ of $D_L$ becomes $2n$ in the real model and $n$
in negative Witt degrees over $\R$.  Thus the torsion bound used in
the preceding proof is also necessary for Hennig's decomposition
over all fields: that decomposition in negative degrees gives free
Witt cohomology.  Equivalently, the general conjecture asks whether
all nonzero Smith factors of $D_L$ are units.

The remaining issue is the primitivity of the global integer maps.
The local projective-space calculation does not force this condition.
Even a matrix with entries $\pm1$, such as
$\left(\begin{smallmatrix}1&1\\1&-1\end{smallmatrix}\right)$,
can have a Smith factor $2$.  After dividing the real differentials
by $2$, if the preceding complexes have free cohomology, a cone of a
Gysin or restriction map has free cohomology precisely when the map
has saturated image on cohomology.  Local unit cancellations do not
establish this: choices on boundary charts must agree, and elimination
changes the remaining maps.  The general case remains open here.

Topologically, the same question asks for the absence of odd torsion
and collapse at $E_2$ of the integral Bockstein spectral sequence.
Its first differential on $\Xi_N=H^*(M_N;\F_2)$ is
$\beta+w_L\smile(-)$; higher differentials detect cyclic summands
of order $2^e$ with $e\geq2$.
\end{remark}